\documentclass[11pt]{amsart}
\usepackage[active]{srcltx}
\usepackage[utf8]{inputenc}
\usepackage{amsmath,mathtools}
\usepackage{stmaryrd}
\usepackage{MnSymbol}
\usepackage{hyperref}
\usepackage{tikz}
\usetikzlibrary{arrows,automata}
\usepackage{enumitem}
\usepackage{mathrsfs}
\usepackage{pst-plot}
\usepackage{auto-pst-pdf}
\usepackage[all]{xy}
\usepackage{stackrel}
\usepackage{listings}
\lstdefinelanguage{GAP}{%
  morekeywords={%
    Assert,Info,IsBound,QUIT,%
    TryNextMethod,Unbind,and,break,%
    continue,do,elif,%
    else,end,false,fi,for,%
    function,if,in,local,%
    mod,not,od,or,%
    quit,rec,repeat,return,%
    then,true,until,while%
  },%
  sensitive,%
  morecomment=[l]\#,%
  morestring=[b]",%
  morestring=[b]',%
}[keywords,comments,strings]
\usepackage[T1]{fontenc}
\usepackage{xcolor}
\usepackage{enumitem}
\newtheorem{thm}{Theorem}[section]
\newtheorem{cor}[thm]{Corollary}
\newtheorem{deff}[thm]{Definition}

\newtheorem{lem}[thm]{Lemma}
\newtheorem{prop}[thm]{Proposition}
\newtheorem{example}[thm]{Example}

\newcommand{\abs}[1]{\left\vert#1\right\vert}

\newcommand{\cc}[1]{\textcolor{red}{{#1}}}

\newcommand{\qq}[1]{\textcolor{magenta}{{#1}}}

\newcommand{\vv}[1]{\textcolor{cyan}{{#1}}}
\def\pv#1{\ensuremath{\mathsf{#1}}}
\newcommand{\J}{\mathrel{\mathscr J}} 
\newcommand{\R}{\mathrel{\mathscr R}} 
\newcommand{\eL}{\mathrel{\mathscr L}} 
\newcommand{\HH}{\mathrel{\mathscr H}}

\newcommand{\DDet}{\mathop{\mathrm{Det}}\nolimits}

\usepackage{caption}

\makeatletter
\newcommand{\mylabel}[2]{#2\def\@currentlabel{#2}\label{#1}}
\makeatother

\begin{document}
\title[The determinant of finite semigroups not in \pv{ECom}]{A step to compute the determinant of finite semigroups not in \pv{ECom}}
\author{M. H. Shahzamanian}
\thanks{ORCID: \texttt{http://orcid.org/0000-0001-6072-8552}}
\address{M. H. Shahzamanian\\ CMUP, Departamento de Matemática, Faculdade de Ciências, Universidade do Porto, Rua do Campo Alegre s/n, 4169--007 Porto (Portugal).}
\email{m.h.shahzamanian@gmail.com}

\begin{abstract}
This paper initiates the study of non-zero determinants arising from finite semigroups containing a pair of non-commuting idempotents. 
We focus on a natural class of such semigroups, termed $\ll$-smooth semigroups, and develop an approach to compute their associated contracted semigroup determinants. 
These computations are motivated by applications to coding theory, in particular, by the potential extension the MacWilliams theorem for codes over semigroup algebras.
Our results provide a structural framework that may serve as a foundation for further exploration of semigroup determinants beyond the pseudovariety \pv{ECom}.
\end{abstract}
\maketitle


\section{Introduction}

In the 1880s, Dedekind introduced the group determinant for finite groups, which he studied extensively with Frobenius. Independently, Smith offered an alternative approach to this concept, as detailed in \cite{Smith}. The group determinant corresponds to the determinant of a $G \times G$ matrix whose $(g,h)$-entry is $x_{gh}$, where $G$ is a finite group and each $x_k$ is a variable indexed by $k \in G$. This foundational idea was later extended to finite semigroups \cite{Lindstr, Wilf, Wood}.
An important application of semigroup determinants arises in coding theory particularly in generalizing the MacWilliams theorem from codes over finite fields to those over chain rings. In this context, linear codes over finite Frobenius rings exhibit the extension property \cite{Wood-Duality}. Crucially, the semigroup determinant plays a central role: it is non-zero if and only if $\mathbb{C}S$ is a Frobenius algebra, as shown in \cite{Ste-Fac-det, Okn}.

Steinberg's work in \cite{Ste-Fac-det} provides a factorization of the semigroup determinant for commutative semigroups, showing it either vanishes or decomposes into linear polynomials. His work builds on prior results of Ponizovski\u{\i}~\cite{Ponizovski} and Wenger~\cite{Wenger} on Frobenius semigroup algebras. Moreover, Steinberg showed that the determinant for inverse semigroups can be computed via the groupoid determinant of an associated finite groupoid.

More recently, in \cite{Sha-Det}, the semigroup determinant is studied within the pseudovariety $\pv{ECom}$, with emphasis on characterizing when such determinants are non-zero and exploring their factorization. This extends the work of Steinberg \cite{Ste-Fac-det}, who focused on semigroups with central idempotents to factor determinants in the commutative case. Notably, by a result of Ash, $\pv{ECom}$ coincides with the pseudovariety generated by finite inverse semigroups, making it a rich and natural setting for generalizing determinant theory in semigroup algebras. Since it properly contains both inverse semigroups and semigroups with central idempotents, $\pv{ECom}$ provides a broader setting for investigating the structural behavior of semigroup determinants.

In this paper, we extend the study of semigroup determinants beyond the pseudovariety $\pv{ECom}$ by investigating finite semigroups containing a pair of non-commutative idempotents. Our primary aim is to understand when the semigroup determinant is non-zero and how it factors in this broader setting. We focus on a particular class of such semigroups satisfying specific structural conditions, laying the groundwork for further exploration of determinants in more complex semigroups.

We define a partial order on finite semigroups whose complex semigroup algebras are unital, extending the natural order on idempotents. Although this order is not necessarily transitive in general, we restrict our attention to semigroups where transitivity holds. These semigroups are classified with respect to this order, and we concentrate on a subclass termed $\ll$-smooth semigroups. Within this subclass, we characterize those with non-zero determinants and examine their factorization behavior.

The paper is organized as follows. Section~\ref{sec:prelim} provides preliminaries on semigroups and semigroup determinants. Section~\ref{llAndlll} introduces the partial order and explores its properties. Section~\ref{DetNotInEcom} focuses on computing determinants for $\ll$-smooth semigroups. Finally, Section~\ref{sec:appendix} presents illustrative examples, with computations carried out using custom C{\fontseries{b}\selectfont\#} programs.


\section{Preliminaries}\label{sec:prelim}

\subsection{Semigroups}

For standard notation and terminology relating to semigroups, we refer the reader to~\cite[Chap. 5]{Alm}, \cite[Chaps. 1-3]{Cli-Pre} and~\cite[Appendix A]{Rho-Ste}.


Let $S$ be a finite semigroup, and let $a, b \in S$. 
The Green’s relations, introduced by Green~\cite{Gre}, are defined as follows:  
\begin{itemize}
    \item $a \R b$ if $aS^1 = bS^1$.  
    \item $a \eL b$ if $S^1a = S^1b$.  
    \item $a \HH b$ if $a \R b$ and $a \eL b$.  
    \item $a \J b$ if $S^1aS^1 = S^1bS^1$.  
\end{itemize}
If $S$ has an identity element, then $S^1 = S$. Otherwise, we define $S^1 = S \cup \{1\}$, where $1$ acts as the identity element. Similarly, for any subset $T \subseteq S$, we define $T^1 = T \cup \{1\}$, where $1$ is the identity of $S^1$.  
The $\R, \eL, \HH,$ and $\J$ classes of an element $a$ are denoted by $R_a, L_a, H_a,$ and $J_a$, respectively.  

Additionally, we define a modified version of the $\eL$ relation. We say that $a \widetilde{\eL} b$ if and only if $a$ and $b$ share the same set of idempotent right identities, meaning $ae = a$ if and only if $be = b$, following the work of Fountain et al.~\cite{Fou-Gom-Gou}. The dual relation $\widetilde{\R}$ is defined analogously, and we set $\widetilde{\HH} = \widetilde{\eL} \wedge \widetilde{\R}$. The equivalence classes of these relations for an element $s$ are denoted by $\widetilde{L}_s, \widetilde{R}_s,$ and $\widetilde{H}_s$, respectively.  
For further results on these topics, we refer the reader to~\cite{Lawson-Sem-Cat}. 


An element $e$ of $S$ is called \emph{idempotent} if $e^2 = e$.  
The set of all idempotents of $S$ is denoted by $E(S)$.  
An idempotent $e$ of $S$ serves as the identity element of the monoid $eSe$.  
The group of units of $eSe$, denoted by $G_e$, is called the \emph{maximal subgroup} of $S$ at $e$.  

A \emph{left ideal} of a semigroup $S$ is a nonempty subset $A \subseteq S$ such that $SA \subseteq A$.  
Dually, a \emph{right ideal} of $S$ is a subset $A$ satisfying $AS \subseteq A$.  
An \emph{ideal} of $S$ is a subset that is both a left and a right ideal.  
Every finite semigroup $S$ has a unique minimal ideal, which is called the \emph{kernel} of $S$.  

A semigroup $S$ is called \emph{inverse} if, for every $s \in S$, there exists a unique element $s^{-1} \in S$ such that  
\[
ss^{-1}s = s \quad \text{and} \quad s^{-1}ss^{-1} = s^{-1}.
\]  
For an element $s \in S$, we define $s^{\omega}$ as the limit of the sequence $(s^{n!})_n$.  


A \emph{pseudovariety} of semigroups is a class of finite semigroups that is closed under taking subsemigroups, homomorphic images, and finite direct products.  
The pseudovariety $\pv{S}$ consists of all finite semigroups, while the pseudovariety $\pv{G}$ is the class of all finite groups.  
Similarly, $\pv{Sl}$ and $\pv{Com}$ denote the pseudovarieties of all finite semilattices and commutative semigroups, respectively.  

The operator $\pv{E}$ associates a pseudovariety $\pv{V}$ to the class of finite semigroups such that the subsemigroup generated by the idempotents of the semigroup belongs to $\pv{V}$.  
This is formally expressed as  
\[
\mathsf{EV} = \{S \in \mathsf{S} \mid \langle E(S)\rangle \in \mathsf{V}\}.
\]  
If a finite semigroup $S$ is a member of $\mathsf{ECom}$, then the subsemigroup generated by the idempotents of $S$ coincides with the set of idempotents of $S$.  
Consequently, the pseudovariety $\pv{ESl}$ is equivalent to the pseudovariety $\pv{ECom}$.  
By a celebrated result of Ash~\cite{Ash}, the pseudovariety generated by finite inverse semigroups is precisely the pseudovariety $\mathsf{ECom}$.  


Let $G$ be a group, and let $n$ and $m$ be integers.  
Consider an $m\times n$ matrix $P = (p_{ji})$ with entries in $G \cup \{0\}$.  
The \emph{Rees matrix semigroup} $\mathcal{M}^{0}(G, n, m; P)$ consists of all triples $(i, g, j)$, where $g \in G$, $1 \leq i \leq n$, and $1 \leq j \leq m$, together with a zero element $0$.  
The multiplication between nonzero elements is defined by  
\[
(i, g, j)(i', g', j') =
\begin{cases}
  (i, g p_{ji'} g', j') & \text{if } p_{ji'} \neq 0, \\
  0 & \text{otherwise},
\end{cases}
\]
for every $(i, g, j), (i', g', j') \in \mathcal{M}^{0}(G, n, m; P)$.  
The Rees matrix semigroup $\mathcal{M}^{0}(G, n, m; P)$ is \emph{regular} if and only if each row and each column of $P$ contains at least one nonzero entry.  
In this case, all nonzero elements of $\mathcal{M}^{0}(G, n, m; P)$ are $\J$-equivalent.  

We denote by $\mathcal{B}_n(G)$ an $n \times n$ Brandt semigroup over a group $G$.  
Note that $\mathcal{B}_n(G)$ is an inverse completely $0$-simple semigroup.


For a semigroup $S$, a \emph{principal series} of $S$ is a chain of ideals  
\[
S = S_1 \supsetneqq S_2 \supsetneqq \cdots \supsetneqq S_{n} \supsetneqq S_{n+1} = \emptyset
\]
such that there is no ideal of $S$ strictly between $S_i$ and $S_{i+1}$.  
(For convenience, we consider the empty set as an ideal of $S$.)  
Each \emph{principal factor} $S_i / S_{i+1}$ ($1 \leq i \leq n$) of $S$ is either completely $0$-simple, completely simple
or null.

Every completely $0$-simple factor is isomorphic to a regular Rees matrix semigroup over a finite group $G$.  
Furthermore, every finite semigroup has a principal series.

\subsection{Incidence Algebras and M{\"o}bius Functions}

Let $(P, \leq)$ be a finite poset. 
The \emph{incidence algebra} $\mathbb{C}\llbracket P\rrbracket$ consists of functions $f: P \times P \to \mathbb{C}$ satisfying $f(x, y) \neq 0 \Rightarrow x \leq y$, with convolution product:
\[
(f \ast g)(x, y) = \sum\limits_{x\leq z\leq y}f(x, z)g(z, y).
\]
The \emph{convolution identity} is the delta function $\delta(x, y)$, which is 1 if $x = y$ and 0 otherwise.

The \emph{zeta function} $\zeta_P$ of $P$ is defined by:
\[
\zeta_P(x,y) =
\begin{cases}
1, & x \leq y, \\
0, & \text{otherwise}.
\end{cases}
\]
It is upper triangular with ones on the diagonal, ensuring the existence of its inverse, the \emph{M{\"o}bius function} $\mu_P$. When $P$ is clear from context, we omit the subscript.


Let $f$ be a function from $P$ to $\mathbb{C}$. By applying M{\"o}bius inversion, if $g$ is the function from $P$ to $\mathbb{C}$ given by $g(x) = \sum\limits_{y \leq x} f(y)$, then
\[
f(x) = \sum\limits_{y \leq x} \mu_P(y, x) g(y),
\]
for every $x \in P$.

We recommend that the reader refer to \cite{IncidenceAlgebras} for further information on this section.

\subsection{Determinant of a semigroup}
For standard notation and terminology related to finite dimensional algebras, refer to \cite{Assem-Ibrahim, Benson}.

A based algebra is a finite dimensional complex algebra $A$ with a distinguished basis $B$. The pair $(A, B)$ represents the algebra and its basis. The multiplication in $A$ is determined by its structure constants with respect to $B$, given by the equation
\[
bb' = \sum_{b'' \in B} c_{b'', b, b'} b'' \quad \text{for} \quad b, b' \in B, \, c_{b'', b, b'} \in \mathbb{C}.
\]
Let $X_B = \{x_b \mid b \in B\}$ be a set of variables corresponding to $B$. These structure constants are represented in the Cayley table, a $B \times B$ matrix with entries in the polynomial ring $\mathbb{C}[X_B]$. The matrix is defined as
\[
C(A, B)_{b, b'} = \sum_{b'' \in B} c_{b'', b, b'} x_{b''}.
\]
The determinant of this matrix, denoted $\theta_{(A, B)}(X_B)$, is either zero or a homogeneous polynomial of degree $|B|$.


Let $S$ be a finite semigroup. The semigroup $\mathbb{C}$-algebra $\mathbb{C}S$ consists of all formal sums $\sum_{s \in S} \lambda_s s$, where $\lambda_s \in \mathbb{C}$, with multiplication defined by
\[
\left( \sum_{s \in S} \lambda_s s \right) \cdot \left( \sum_{t \in S} \mu_t t \right) = \sum_{u = st \in S} \lambda_s \mu_t u.
\]
$\mathbb{C}S$ is a finite-dimensional $\mathbb{C}$-algebra with basis $S$. If $A = \mathbb{C}S$ and $B = S$, the Cayley table $C(S) = C(\mathbb{C}S, S)$ is an $S \times S$ matrix over $\mathbb{C}[X_S]$, with
\[
C(S)_{s, s'} = x_{ss'},
\]
where $X_S = \{x_s \mid s \in S\}$ is a set of variables corresponding to $S$. 
The determinant $\DDet C(\mathbb{C}S, S)$ is denoted by $\theta_S(X_S)$ and is called the (Dedekind-Frobenius) \emph{semigroup determinant} of $S$. If $S$ is fixed, we write $X$ instead of $X_S$. 
For more information, see \cite{Frobenius1903theorie}, \cite[Chapter~16]{Okn}, and \cite{Ste-Fac-det}.

The \emph{contracted semigroup algebra} of a semigroup $S$ with a zero element $0$ over the complex numbers is defined as $\mathbb{C}_0S = \mathbb{C}S / \mathbb{C}0$, where $\mathbb{C}0$ is a one-dimensional two-sided ideal. 
This algebra has a basis consisting of the non-zero elements of $S$, with multiplication extending that of $S$, and the zero of the semigroup identified with the zero of the algebra. The contracted semigroup determinant of $S$, denoted by $\widetilde{\theta}_S$, is the determinant of $\widetilde{C}(S) = C(\mathbb{C}_0S, S \setminus \{0\})$, where $\widetilde{C}(S)_{s,t} = x_{st}$ if $st \neq 0$, and 0 otherwise. We use $\widetilde{X} = X_{S \setminus \{0\}}$ if $S$ is understood.


%

According to Proposition 2.7 in~\cite{Ste-Fac-det} (also mentioned in~\cite{Wood}), there is a connection between the contracted semigroup determinant and the semigroup determinant of a semigroup $S$ with a zero element. There exists a $\mathbb{C}$-algebra isomorphism between $\mathbb{C}S$ and the product algebra $\mathbb{C}_0S \times \mathbb{C}0$, mapping $s \in S$ to $(s,0)$. Define $y_s = x_s - x_0$ for $s \neq 0$, and let $Y = \{y_s \mid s \in S \setminus \{0\}\}$. Then, we have the relation $\theta_S(X) = x_0 \widetilde{\theta}_S(Y)$. Consequently, $\widetilde{\theta}_S(\widetilde{X})$ can be obtained from $\theta_S(X)/x_0$ by replacing $x_0$ with 0.


\section{Relations \(\ll\) and \(\lll\)}\label{llAndlll}

Necessary conditions for a semigroup $S$ to have a non-zero $\theta_S(X)$ are stated in \cite{Ste-Fac-det}.
If $\theta_S(X)$ is not equal to 0, then the semigroup algebra $\mathbb{C}S$ is a unital algebra, according to Theorem 2.1 in~\cite{Ste-Fac-det}.

We define the functions $\varphi^{\ast}$ and $\varphi^{+}$ from $S$ to the power set of $E(S)$ 
as follows:
\[\varphi^{\ast}(s)=\{e\in E(S)\mid se = s\}\ \text{and} \ \varphi^{+}(s)=\{e\in E(S)\mid es = s\}.\]
If $S$ is finite and $\mathbb{C}S$ is a unital algebra, then the subsets $\varphi^{\ast}(s)$ and $\varphi^{+}(s)$ are nonempty, for every $s\in S$ (\cite[Lemma 3.1]{Sha-Det}).

In~\cite{Sha-Det}, the determinant of a semigroup within the pseudovariety \pv{ECom} is explored.
In this paper, we take one step further and begin investigating the determinant of a semigroup not within the pseudovariety \pv{ECom}.
Throughout the paper, we consider a finite semigroup 
$S$ with the assumption that the semigroup algebra $\mathbb{C}S$ is a unital algebra.

Let $s\in S$. Since $\mathbb{C}S$ is a unital algebra,
the subsets $\varphi^{\ast}(s)$ and $\varphi^{+}(s)$ are nonempty.
We denote the kernel of $\langle \varphi^{\ast}(s) \rangle$ and $\langle \varphi^{+}(s) \rangle$ by $s^{\ast\ast}$ and $s^{++}$, respectively.

Note that $s^{\ast\ast} = t^{\ast\ast}$ if and only if $\varphi^{\ast}(s)=\varphi^{\ast}(t)$. 
Indeed, suppose that $s^{\ast\ast} = t^{\ast\ast}$ and $e\in \varphi^{\ast}(s)$. 
Hence, we have $se=s$. 
Let $f \in t^{\ast\ast}$. It is easy to verify that $tf=t$.
Then, we have $te=(tf)e=t(fe)$.
Since $f \in s^{\ast\ast} (= t^{\ast\ast})$ and $e\in \varphi^{\ast}(s)$,
we have $f'=fe \in s^{\ast\ast}=t^{\ast\ast}$.
Hence $te=tf'=t$. 
Then, $e\in \varphi^{\ast}(t)$ and, thus, we have $\varphi^{\ast}(s)=\varphi^{\ast}(t)$. 
Also, we have $s^{++} = t^{++}$ if and only if $\varphi^{+}(s)=\varphi^{+}(t)$. 
Then, the equivalence relations $\widetilde{\eL}$, $\widetilde{\R}$
and $\widetilde{\HH}$ can be described as follows:
\begin{enumerate}
\item $s \widetilde{\eL} t$ if $s^{\ast\ast} = t^{\ast\ast}$;
\item $s \widetilde{\R} t$ if $s^{++} = t^{++}$;
\item $s \widetilde{\HH} t$ if $s^{\ast\ast} = t^{\ast\ast}$ and $s^{++} = t^{++}$.
\end{enumerate}

It is clear that if $e$ is an idempotent then $e \in e^{\ast\ast}, e^{++}$.

Let $s$ and $t$ be elements of $S$. We define $s \ll t$ if \[s = s^{++}ts^{\ast\ast}.\]

We say that a semigroup $S$ is singleton-rich
if the cardinality of subsets $s^{\ast\ast}$ and $s^{++}$ is equal to one, for every $s$ in $S$.
In this case, we denote the single element of the subsets $s^{\ast\ast}$ and $s^{++}$ by $s^{\ast}$ and $s^{+}$, respectively.
Note that, it is easy to verify that $s^{\ast}$ and $s^{+}$ are idempotent. 
In this paper, we assume that all semigroups are singleton-rich.

\begin{lem}\label{eslls}
The following statements hold:
\begin{enumerate}
\item
Let $s\in S$ and $e, f\in E(S)$.
We have $es, se, esf \ll s$.
\item Let $s_1,s_2\in S$. Then, we have $(s_1s_2)^{\ast} \leq s_2^{\ast}$ and $(s_1s_2)^{+} \leq s_1^{+}$.
\item Let $e,f\in E(S)$. If $ef=e$ or $fe=e$, then we have $e\leq f$.
\end{enumerate}

\end{lem}

\begin{proof}
(1) Since $e\in \varphi^{+}(es)$, we have $(es)^{+}e=(es)^{+}$.
Then, we deduce \[(es)^{+}s(es)^{\ast}=(es)^{+}es(es)^{\ast}=es.\]
It follows that $es \ll s$. 

Similarly, we have $se,esf \ll s$.
%

(2) Since $s_1s_2s_2^{\ast}=s_1s_2$, we have $s_2^{\ast}\in \varphi^{\ast}(s_1s_2)$, and thus, we get that $(s_1s_2)^{\ast} \leq s_2^{\ast}$.

Similarly, we have $(s_1s_2)^{+} \leq s_1^{+}$.

(3) As $e^{+}=e$, if $fe=e$, then we have $e\leq f$.

Similarly, if $ef=e$, then we have $e\leq f$.
\end{proof}

\begin{prop}\label{lllES}
The following conditions hold:
\begin{enumerate}
\item $s \ll t$ if and only if $s \in E(S)^1tE(S)^1$.
\item there do not exist pairwise distinct elements $s_1,\ldots,s_n$ with $1 < n$ such that
\[s_1 \ll s_2 \ll \cdots \ll s_n \ll s_1.\]
\end{enumerate}
\end{prop}

\begin{proof}
(1) If $s \ll t$ then, clearly we have $s \in E(S)^1tE(S)^1$.

Now, we suppose that $s \in E(S)^1tE(S)^1$.
Then, there exists elements $e, f \in E(S)^1$ such that $s = etf$.
It is easy to verify that $s^{+}e \in \langle\varphi^{+}(s)\rangle$ and $fs^{\ast}\in \langle\varphi^{\ast}(s)\rangle$.
Since $s^{\ast\ast}$ and $s^{++}$ are the kernel of $\langle \varphi^{\ast}(s) \rangle$ and $\langle \varphi^{+}(s) \rangle$, respectively, and  
$s^{\ast\ast}$ and $s^{++}$ have only one element, 
we have $s^{+}e=s^{+}$ and $fs^{\ast}=s^{\ast}$.
Therefore, we have $s = s^{+}ts^{\ast}$.

(2)
We assume the contrary that there exist pairwise distinct elements $s_1, \ldots, s_n$ with $1 < n$ such that
\[s_1 \ll s_2 \ll \cdots \ll s_n \ll s_1.\]
We have $s_i = s_i^{+}s_{i+1}s_i^{\ast}$, for $1\leq i< n$, and $s_n = s_n^{+}s_1s_n^{\ast}$.
Then, we get that 
\[s_1 = s_1^{+} s_2^{+} \cdots s_n^{+} s_1 s_n^{\ast} \cdots s_2^{\ast} s_1^{\ast}. \]
It follows that
\[s_1 = (s_1^{+} s_2^{+} \cdots s_n^{+})^{\omega} s_1 (s_n^{\ast} \cdots s_2^{\ast} s_1^{\ast})^{\omega}\]
and, thus, we have \(s_1(s_n^{\ast} \cdots s_2^{\ast} s_1^{\ast})^{\omega} =s_1.\)
Hence, we get that $(s_n^{\ast} \cdots s_2^{\ast} s_1^{\ast})^{\omega}\in \varphi^{\ast}(s_1)$.
As \[(s_n^{\ast} \cdots s_2^{\ast} s_1^{\ast})^{\omega}s_1^{\ast}=(s_n^{\ast} \cdots s_2^{\ast} s_1^{\ast})^{\omega}\] and $s_1^{\ast\ast}$ is the kernel of $\langle \varphi^{\ast}(s_1) \rangle$, 
we have \[(s_n^{\ast} \cdots s_2^{\ast} s_1^{\ast})^{\omega}=s_1^{\ast}.\]
Similarly, we have $(s_{i-1}^{\ast}\cdots s_1^{\ast}s_n^{\ast} \cdots s_{i+1}^{\ast}s_i^{\ast})^{\omega}=s_i^{\ast}$, for every $2 < i \leq n$.
Therefore, we get that
\[s_1^{\ast}=s_1^{\ast}s_1^{\ast}
=s_1^{\ast}(s_n^{\ast} \cdots s_2^{\ast} s_1^{\ast})^{\omega}
=(s_1^{\ast}s_n^{\ast} \cdots s_2^{\ast})^{\omega} s_1^{\ast}
=s_2^{\ast}s_1^{\ast},\]
\begin{align*}
s_i^{\ast}
=&s_i^{\ast}s_i^{\ast}
=s_i^{\ast}(s_{i-1}^{\ast}\cdots s_1^{\ast}s_n^{\ast} \cdots s_{i+1}^{\ast}s_i^{\ast})^{\omega}\\
=&(s_i^{\ast}s_{i-1}^{\ast}\cdots s_1^{\ast}s_n^{\ast} \cdots s_{i+1}^{\ast})^{\omega}s_i^{\ast}
=s_{i+1}^{\ast}s_i^{\ast}, 
\end{align*}
for every $2 < i < n$,
and 
$s_n^{\ast}=s_{1}^{\ast}s_n^{\ast}$.
Now, by the last equalities, we have
\[s_1^{\ast}s_2^{\ast}=s_1^{\ast}s_n^{\ast}\cdots s_3^{\ast}s_2^{\ast}=s_n^{\ast}\cdots s_3^{\ast}s_2^{\ast}=s_2^{\ast}.\]
Hence, we have $s_2^{\ast}s_1^{\ast}=s_1^{\ast}$ and $s_1^{\ast}s_2^{\ast}=s_2^{\ast}$.
Now, as $(s_1^{\ast})^{+}=s_1^{\ast}$, $s_2^{\ast}\in \varphi^{+}(s_1^{\ast})$ and $s_1^{\ast}s_2^{\ast}=s_2^{\ast}$, we have $s_1^{\ast}=s_2^{\ast}$.
Similarly, we get that  $s_1^{+}=s_2^{+}$.
Therefore, we have $s_1 = s_1^{+}s_2s_1^{\ast}=s_2^{+}s_2s_2^{\ast}=s_2$. A contradiction that $s_1$ and $s_2$ are distinct.
\end{proof}

Since $s = s^{+}ss^{\ast}$, for every $s\in S$, we have $s\ll s$ and, thus, the relation $\ll$ is reflexive.
Also, by part (2) of Proposition~\ref{lllES}, the relation $\ll$ is antisymmetric. 
However, as illustrated in Example~\ref{Ex01}, the relation $\ll$ may not be transitive. 
Therefore, we define $\lll$ as the smallest partially ordered set containing $\ll$.

Let $Z\colon \mathbb{C}S \rightarrow \mathbb{C}S$ be a map given by $Z(s) =\sum\limits_{s'\lll s}s'$ on $s \in S$ with a linear extension.
By applying M\"{o}bius inversion, we can establish an inverse for $Z$,
making it bijective. As mentioned in the following proposition.

\begin{prop}\label{Z}
The mapping $Z$ is bijective. 
\end{prop}

%


\section{Sequences \(\varphi\) and \(\psi\)}\label{varphiAnspsi}

As assumed in the previous section, we are working with a finite singleton-rich semigroup $S$.

For $s$ and $t$ in $S$,
one recursively defines two sequences $s_i$ and $t_i$ by
$$s_0=s, t_0=t$$
and
\[
\begin{cases}
s_{i+1}=s_it_i^{+}\ \text{and}\ t_{i+1}=s_i^{\ast}t_i &\text{if}\ s_i^{\ast}t_i^{+}=t_i^{+}s_i^{\ast};\\
s_{i+1}=s_i\ \text{and}\ t_{i+1}=t_i                         &\text{else}.\\
\end{cases}
\]
We could define the mapping \[\varepsilon\colon S\times S \rightarrow E(S) \times E(S)\] as follows:
\begin{align*}
\varepsilon^{(s,t)} 
= 
& 
(s_i^{\ast},t_i^{+}),\  \text{where}\ i\ \text{is an integer such that}\ s_i = s_{i+1}\ \text{and}\ t_i = t_{i+1}.
\end{align*}
In the case where \(s_i^{\ast}=t_i^{+}\), for convenience, we denote \(\varepsilon^{(s,t)}\) as \(s_i^{\ast}\).
If $s_i^{\ast}t_i^{+}=t_i^{+}s_i^{\ast}$, for all $0\leq i$, then both sequences $s_i$ and $t_i$ converges to an equal element in $S$.

We define functions $\varphi$ and $\psi$ on $S\times S$ as follows:
\[\varphi(s,t)=(s^{\ast}t)^{+} \ \text{and}\ \psi(s,t)=(st^{+})^{\ast},\]
for every $s,t\in S$.

Now, for $s$ and $t$ in $S$, based on the functions $\varphi$ and $\psi$, we define the sequences $s^{\varphi}_i,t^{\varphi}_i,s^{\psi}_i$ and $t^{\psi}_i$ as follows:
\begin{center}
\begin{tabular}{ll}
\begin{tabular}{ll}
$s^{\varphi}_0=s$, & $t^{\varphi}_0=t;$\\
\text{and} &\\
$s^{\varphi}_{i+1}=s^{\varphi}_i\varphi(s^{\varphi}_i,t^{\varphi}_i)$, & $t^{\varphi}_{i+1} = {s^{\varphi}_i}^{\ast}t^{\varphi}_i$;
\end{tabular}
&
\begin{tabular}{ll}
$s^{\psi}_0=s$, & $t^{\psi}_0=t;$\\
 &\\
$s^{\psi}_{i+1}=s^{\psi}_i{t^{\psi}_i}^{+}$, & $t^{\psi}_{i+1}=\psi(s^{\psi}_i,t^{\psi}_i)t^{\psi}_i$.
\end{tabular}
\end{tabular}
\end{center}

\begin{lem}\label{snrtnl}
Let $s,t\in S$ and
let $m$ be an integer such that
$s_i^{\ast}t_i^{+}=t_i^{+}s_i^{\ast}$, for all $i<m$.
Then, we have
\begin{enumerate}
\item $s_{i+1}^{\ast}, t_{i+1}^{+}\leq s_i^{\ast}, t_i^{+}$, 
\item $s_{i+1}=st_i^{+}$ and $t_{i+1}=s_i^{\ast}t$, 
\end{enumerate}
for all $i<m$.
\end{lem}

\begin{proof}
(1)
Since
\[s_{i+1}s_i^{\ast} = s_it_i^{+}s_i^{\ast} = s_is_i^{\ast}t_i^{+} = s_it_i^{+} = s_{i+1},\]
we have $s_i^{\ast} \in\varphi^{\ast}(s_{i+1})$. Now, as $s_{i+1}^{\ast\ast}$ has only one element, we have $s_{i+1}^{\ast}\leq s_i^{\ast}$.
Also, as
\[s_{i+1}t_i^{+} = s_it_i^{+}t_i^{+} = s_it_i^{+} = s_{i+1},\]
we have $s_{i+1}^{\ast}\leq t_i^{+}$. Similarly, we have $t_{i+1}^{+}\leq s_i^{\ast}, t_i^{+}$.

(2)
We prove that $s_{i+1}=st_i^{+}$ and $t_{i+1}=s_i^{\ast}t$, by induction on $i$.
The case $i = 0$ is clear.
Assume then that $i > 0$, and that the result holds for smaller values of $i$.
We have $s_{i+1} = s_it_i^{+}$. By hypothesis of induction, we have $s_i = st_{i-1}^{+}$. 
It follows that $s_{i+1} = st_{i-1}^{+}t_i^{+}$. Now, as $t_i^{+}\leq t_{i-1}^{+}$ (by part (1)),
we have $s_{i+1} = st_i^{+}$.
Similarly, we have $t_{i+1}=s_i^{\ast}t$.
\end{proof}

\begin{lem}\label{sirtil}
For elements $s$ and $t$ in $S$, we have 
\begin{enumerate}
\item
$\varphi(s^{\varphi}_i,t^{\varphi}_i) \leq {s^{\varphi}_i}^{\ast}$ and $\psi(s^{\psi}_i,t^{\psi}_i) \leq {t^{\psi}_i}^{+}$;
\item 
$\varphi(s^{\varphi}_{i+1},t^{\varphi}_{i+1}) \leq \varphi(s^{\varphi}_i,t^{\varphi}_i)$ and $\psi(s^{\psi}_{i+1},t^{\psi}_{i+1}) \leq \psi(s^{\psi}_i,t^{\psi}_i)$, for all $0 \leq i$;
\item 
$s^{\varphi}_{i+1}=s\varphi(s^{\varphi}_i,t^{\varphi}_i)$ 
and 
$t^{\varphi}_{i+1}=\varphi(s^{\varphi}_i,t^{\varphi}_i)t$,
for all $0 \leq i$;
\item 
$s^{\psi}_{i+1}=s\psi(s^{\psi}_i,t^{\psi}_i)$ 
and 
$t^{\psi}_{i+1}=\psi(s^{\psi}_i,t^{\psi}_i)t$,
for all $0 \leq i$.
\item
$\varphi(s^{\varphi}_i,t^{\varphi}_i)=\varphi(s^{\varphi}_i,t)$ and $\psi(s^{\psi}_i,t^{\psi}_i)=\psi(s,t^{\psi}_i)$, for every $0 \leq i$.
\end{enumerate}
\end{lem}

\begin{proof}
(1)
As \({s^{\varphi}_i}^{\ast}{s^{\varphi}_i}^{\ast}t^{\varphi}_i={s^{\varphi}_i}^{\ast}t^{\varphi}_i\), we have \(\varphi(s^{\varphi}_i,t^{\varphi}_i) = ({s^{\varphi}_i}^{\ast}t^{\varphi}_i)^{+} \leq {s^{\varphi}_i}^{\ast}\).

Similarly, we have $\psi(s^{\varphi}_i,t^{\varphi}_i) \leq {t^{\varphi}_i}^{+}$.

(2)
Since $\varphi(s^{\varphi}_i,t^{\varphi}_i)$ is an idempotent, we have $\varphi(s^{\varphi}_i,t^{\varphi}_i)\in \varphi^{\ast}(s^{\varphi}_i\varphi(s^{\varphi}_i,t^{\varphi}_i))$.
Then, we deduce 
\[
\varphi(s^{\varphi}_i,t^{\varphi}_i)(s^{\varphi}_i\varphi(s^{\varphi}_i,t^{\varphi}_i))^{\ast}{s^{\varphi}_i}^{\ast}t^{\varphi}_i = (s^{\varphi}_i\varphi(s^{\varphi}_i,t^{\varphi}_i))^{\ast}{s^{\varphi}_i}^{\ast}t^{\varphi}_i.
\]
It follows that
\[
\varphi(s^{\varphi}_i,t^{\varphi}_i) \in \varphi^{+}((s^{\varphi}_i\varphi(s^{\varphi}_i,t^{\varphi}_i))^{\ast}{s^{\varphi}_i}^{\ast}t^{\varphi}_i)
\]
and, thus,
\[ 
\varphi(s^{\varphi}_{i+1},t^{\varphi}_{i+1})
=
({s^{\varphi}_{i+1}}^{\ast}t^{\varphi}_{i+1})^{+}
=
((s^{\varphi}_i\varphi(s^{\varphi}_i,t^{\varphi}_i))^{\ast}{s^{\varphi}_i}^{\ast}t^{\varphi}_i)^{+}
\leq
\varphi(s^{\varphi}_i,t^{\varphi}_i).
\]

Similarly, we have $\psi(s^{\psi}_{i+1},t^{\psi}_{i+1}) \leq \psi(s^{\psi}_i,t^{\psi}_i)$.

(3)
We prove that $s^{\varphi}_{i+1}=s\varphi(s^{\varphi}_i,t^{\varphi}_i)$, by induction on $i$.
The case $i = 0$ is clear.
Assume then that $i > 0$, and that the result holds for smaller values of $i$.
We have $s^{\varphi}_{i+1} =s^{\varphi}_i\varphi(s^{\varphi}_i,t^{\varphi}_i)$. By hypothesis of induction, we have $s^{\varphi}_{i+1} =s\varphi(s^{\varphi}_{i-1},t^{\varphi}_{i-1})\varphi(s^{\varphi}_i,t^{\varphi}_i)$. 
Now, by part (2), we have $s^{\varphi}_{i+1} =s\varphi(s^{\varphi}_i,t^{\varphi}_i)$. 

Also, we prove that $t^{\varphi}_{i+1}=\varphi(s^{\varphi}_i,t^{\varphi}_i)t$, by induction on $i$.
We have $t^{\varphi}_{i+1} = {s^{\varphi}_i}^{\ast}t^{\varphi}_i = ({s^{\varphi}_i}^{\ast}t^{\varphi}_i)^{+}{s^{\varphi}_i}^{\ast}t^{\varphi}_i$.
Then, as $({s^{\varphi}_i}^{\ast}t^{\varphi}_i)^{+} \leq {s^{\varphi}_i}^{\ast}$ and by the assumption of the induction, we have
\[t^{\varphi}_{i+1} = \varphi(s^{\varphi}_i,t^{\varphi}_i)\varphi(s^{\varphi}_{i-1},t^{\varphi}_{i-1})t.\]
Now, by part (2), as $\varphi(s^{\varphi}_{i},t^{\varphi}_{i}) \leq \varphi(s^{\varphi}_{i-1},t^{\varphi}_{i-1})$, the result follows.

(4)
Similar to part (3), part (4) also holds.

(5)
The case $i = 0$ is clear.
Assume then that $i > 0$.
By part (3), we have 
\[
\varphi(s^{\varphi}_i,t^{\varphi}_i)
=
({s^{\varphi}_i}^{\ast}t^{\varphi}_i)^{+}
=
({s^{\varphi}_i}^{\ast}\varphi(s^{\varphi}_{i-1},t^{\varphi}_{i-1})t)^{+}.
\]
Since $s^{\varphi}_i= s\varphi(s^{\varphi}_{i-1},t^{\varphi}_{i-1})$, we deduce ${s^{\varphi}_i}^{\ast}\varphi(s^{\varphi}_{i-1},t^{\varphi}_{i-1}) = {s^{\varphi}_i}^{\ast}$ and, thus, 
\[\varphi(s^{\varphi}_i,t^{\varphi}_i) = ({s^{\varphi}_i}^{\ast}t)^{+} = \varphi(s^{\varphi}_i,t).\]

Similarly, we have $\psi(s^{\psi}_i,t^{\psi}_i)=\psi(s,t^{\psi}_i)$.
\end{proof}

Since $S$ is finite, we consider the following descending chain:
\[S= S_1 \supsetneqq S_2 \supsetneqq \cdots \supsetneqq S_{l} \supsetneqq S_{l+1} = \emptyset\]
which forms a principal series for $S$ within this section.

\begin{lem}\label{SiSi1Brandt}
Let $1 \leq i\leq l$. 
If $S_i/S_{i+1}$ is not null, then $S_i/S_{i+1}$ is an inverse completely $0$-simple semigroup or a group.
Moreover, $S_i/S_{i+1}$ is isomorphic with an $n\times n$ Brandt semigroup $\mathcal{B}_{n}(G)$ over a group $G$.
\end{lem}

\begin{proof}
Since $S_i/S_{i+1}$ is not null, 
$S_i/S_{i+1}$ is isomorphic with a regular Rees matrix semigroup $\mathcal{M}^{0}(G, n,m;P)$ or $\mathcal{M}(G, n,m;P)$.
If $S_i/S_{i+1}$ is not inverse, then one or both of the following conditions hold:
\begin{enumerate}
\item there exist integers $1\leq k_1,k_2\leq n$ and $1\leq j \leq m$ such that $k_1 \neq k_2$ and $p_{jk_1},p_{jk_2}\neq 0$;
\item there exist integers $1\leq k\leq n$ and $1\leq j_1,j_2 \leq m$ such that $j_1 \neq j_2$ and $p_{j_1k},p_{j_2k}\neq 0$. 
\end{enumerate}
By symmetry, we may assume the first case.
As the elements $(k_1,p^{-1}_{jk_1},j)$ and $(k_2,p^{-1}_{jk_2},j)$ 
are in the kernel of the subset $\langle\varphi^{\ast}(k_1,p^{-1}_{jk_1},j)\rangle$,
this contradicts the section's assumption that the subset $(k_1,p^{-1}_{jk_1},j)^{\ast\ast}$ contains only one element.

The result follows.
\end{proof}

\begin{lem}\label{ConvergephiAndpsi}
Let $s,t\in S$. There exist integers $i^{\circ}$ and $j^{\circ}$ such that 
$\varphi(s^{\varphi}_i,t^{\varphi}_i) = \varphi(s^{\varphi}_{i^{\circ}},t^{\varphi}_{i^{\circ}})$, for every $i \geq i^{\circ}$,
and 
$\psi(s^{\psi}_j,t^{\psi}_j) = \psi(s^{\psi}_{j^{\circ}},t^{\psi}_{j^{\circ}})$, for every $j\geq j^{\circ}$. 
\end{lem}

\begin{proof}
Since $S$ is finite, there exist integers $m$ and $i^{\circ}$ such that \[\varphi(s^{\varphi}_i,t^{\varphi}_i) \in S_m\setminus S_{m+1},\] for every $i\geq i^{\circ}$.

As the element $\varphi(s^{\varphi}_{i^{\circ}},t^{\varphi}_{i^{\circ}})$ is idempotent, 
by Lemma~\ref{SiSi1Brandt}, 
$S_m/S_{m+1}$ is isomorphic with an $n\times n$ Brandt semigroup $\mathcal{B}_n(G)$ over a group $G$.
Then, $\varphi(s^{\varphi}_{i^{\circ}},t^{\varphi}_{i^{\circ}}) = (k_{i^{\circ}},1_G,k_{i^{\circ}})$ and $\varphi(s^{\varphi}_{i^{\circ}+1},t^{\varphi}_{i^{\circ}+1}) = (k_{i^{\circ}+1},1_G,k_{i^{\circ}+1})$, for some elements $(k_{i^{\circ}},1_G,k_{i^{\circ}}), (k_{i^{\circ}+1},1_G,k_{i^{\circ}+1})\in \mathcal{B}_n(G)$.

By Lemma~\ref{sirtil}.(2), we have $\varphi(s^{\varphi}_{i^{\circ}+1},t^{\varphi}_{i^{\circ}+1}) \leq \varphi(s^{\varphi}_{i^{\circ}},t^{\varphi}_{i^{\circ}})$.
Then, we conclude that $k_{i^{\circ}}=k_{i^{\circ}+1}$ and, thus 
$\varphi(s^{\varphi}_{i^{\circ}},t^{\varphi}_{i^{\circ}}) = \varphi(s^{\varphi}_{i^{\circ}+1},t^{\varphi}_{i^{\circ}+1})$.
Using the same argument, we find that $\varphi(s^{\varphi}_i,t^{\varphi}_i) = \varphi(s^{\varphi}_{i^{\circ}},t^{\varphi}_{i^{\circ}})$, for every $i \geq i^{\circ}$.

Similarly, we can demonstrate that there exists an integer $j^{\circ}$ such that $\psi(s^{\psi}_j,t^{\psi}_j) = \psi(s^{\psi}_{j^{\circ}},t^{\psi}_{j^{\circ}})$, for every $j\geq j^{\circ}$. 
\end{proof}

Then, by Lemma~\ref{ConvergephiAndpsi}, the sequence $\varphi(s^{\varphi}_i,t^{\varphi}_i)$ converges to an equal element. As well for the sequence $\psi(s^{\psi}_j,t^{\psi}_j)$.
We denote the converge of these sequences by $\varphi^{(s,t)}$ and $\psi^{(s,t)}$, respectively.

\begin{cor}\label{sirtilCor}
We have 
\(\varphi^{(s,t)} \leq s^{\ast}$ and $\psi^{(s,t)} \leq t^{+}\), for every $s,t\in S$.
\end{cor}

\begin{proof}
The result follows from Lemma~\ref{sirtil}, relying on parts (1) and (2).
\end{proof}

Note that it may be that \(\varphi^{(s,t)} \not\leq s^{+}$ or $\psi^{(s,t)} \not\leq t^{\ast}\) (see Example~\ref{Ex00}).

\begin{lem}\label{ephiAndpsi}
Let $s,t \in S$.
We have $st = s\varphi^{(s,t)}t = s\psi^{(s,t)}t$.
\end{lem}

\begin{proof}
We prove that $st = s\varphi(s^{\varphi}_i,t^{\varphi}_i)t$, by induction on $i$.

First, we have 
\[
st
=
ss^{\ast}t
=
s(s^{\ast}t)^{+}s^{\ast}t
=
s(s^{\ast}t)^{+}t
=
s\varphi(s,t)t
= 
s\varphi(s^{\varphi}_0,t^{\varphi}_0)t.
\]

Now, assume then that $0 \leq i'$, and $st = s\varphi(s^{\varphi}_i,t^{\varphi}_i)t$, for every $i \leq i'$.
Now, by using Lemma~\ref{sirtil}, we have
\begin{align*}
st 
&
=
s\varphi(s^{\varphi}_{i'},t^{\varphi}_{i'})t
=
s^{\varphi}_{i'+1}t^{\varphi}_{i'+1}
=
s^{\varphi}_{i'+1}{s^{\varphi}_{i'+1}}^{\ast}t^{\varphi}_{i'+1}
=
s^{\varphi}_{i'+1}({s^{\varphi}_{i'+1}}^{\ast}t^{\varphi}_{i'+1})^{+}{s^{\varphi}_{i'+1}}^{\ast}t^{\varphi}_{i'+1}\\
&
=
s^{\varphi}_{i'+1}({s^{\varphi}_{i'+1}}^{\ast}{t^{\varphi}_{i'+1})}^{+}t^{\varphi}_{i'+1}
=
s^{\varphi}_{i'+1}\varphi(s^{\varphi}_{i'+1},t^{\varphi}_{i'+1})t^{\varphi}_{i'+1}\\
&
=
s \varphi(s^{\varphi}_{i'},t^{\varphi}_{i'}) \varphi(s^{\varphi}_{i'+1},t^{\varphi}_{i'+1}) \varphi(s^{\varphi}_{i'},t^{\varphi}_{i'}) t
=
s\varphi(s^{\varphi}_{i'+1},t^{\varphi}_{i'+1})t.
\end{align*}

It follows that $st = s\varphi^{(s,t)}t$.

Similarly, we have $st = s\psi^{(s,t)}t$.
\end{proof}

\begin{lem}\label{phiplusstarequal}
We have 
\[
(s\varphi^{(s,t)})^{\ast} = (\varphi^{(s,t)}t)^{+} = \varphi^{(s,t)},\
\text{and}\ 
(s\psi^{(s,t)})^{\ast} = (\psi^{(s,t)}t)^{+} = \psi^{(s,t)},\]
for every $s,t$ in $S$.
\end{lem}

\begin{proof}
By Lemma~\ref{ConvergephiAndpsi},
there exist an integers $i^{\circ}$
such that 
$\varphi(s^{\varphi}_i,t^{\varphi}_i) = \varphi^{(s,t)}$, for every $i \geq i^{\circ}$.

Since
\[
\varphi^{(s,t)}
=
\varphi(s^{\varphi}_{i^{\circ}+1},t^{\varphi}_{i^{\circ}+1})
=
({s^{\varphi}_{i^{\circ}+1}}^{\ast}t^{\varphi}_{i^{\circ}+1})^{+}
\leq
{s^{\varphi}_{i^{\circ}+1}}^{\ast}
=
(s\varphi(s^{\varphi}_{i^{\circ}},t^{\varphi}_{i^{\circ}}))^{\ast}
=
(s\varphi^{(s,t)})^{\ast}
\]
and
$
(s\varphi^{(s,t)})^{\ast}
\leq
\varphi^{(s,t)}$,
we have $(s\varphi^{(s,t)})^{\ast} = \varphi^{(s,t)}$.

Also, we have
\[
\varphi^{(s,t)} 
= 
\varphi(s^{\varphi}_{i^{\circ}},t^{\varphi}_{i^{\circ}})
= 
({s^{\varphi}_{i^{\circ}}}^{\ast}t^{\varphi}_{i^{\circ}})^{+}
= 
(t^{\varphi}_{i^{\circ}+1})^{+}
= 
(\varphi(s^{\varphi}_{i^{\circ}},t^{\varphi}_{i^{\circ}})t)^{+}
= 
(\varphi^{(s,t)}t)^{+}.
\]

Similarity, we get that
$(s\psi^{(s,t)})^{\ast} = (\psi^{(s,t)}t)^{+} = \psi^{(s,t)}$.
\end{proof}

By Lemma~\ref{phiplusstarequal}, we derive the following corollary.

\begin{cor}\label{phiplusstarequalCor}
We have 
\(\varphi^{(s\varphi^{(s,t)}, \varphi^{(s,t)}t)}=\varphi^{(s,t)}\ 
\text{and}\
\psi^{(s\psi^{(s,t)},\psi^{(s,t)}t)} = \psi^{(s,t)}\), for every $s,t$ in $S$.
\end{cor}

\begin{lem}\label{ephiAndpsiECom}
Let $s,t\in S$ with $s_i^{\ast}t_i^{+}=t_i^{+}s_i^{\ast}$, for every $0\leq i$.
We have $\varphi^{(s,t)} = \psi^{(s,t)} = \varepsilon^{(s,t)}$.
\end{lem}

\begin{proof}
We prove that $s_{2i}=s^{\varphi}_i$ and $t_{2i}=t^{\psi}_i$, by induction on $i$.
The case $i = 0$ is clear.

Assume then that $i > 0$, and that the result holds for smaller values of $i$.
By using Lemma~\ref{snrtnl}.(2), we have 
\[s_{2i}=st_{2i-1}^{+}=s(s_{2i-2}^{\ast}t)^{+}=s({s^{\varphi}_{i-1}}^{\ast}t)^{+}=s\varphi(s^{\varphi}_{i-1},t).\]
By Lemma~\ref{sirtil}.(5), as $\varphi(s^{\varphi}_{i-1},t)=\varphi(s^{\varphi}_{i-1},t^{\varphi}_{i-1})$, we have $s_{2i}=s^{\varphi}_i$.

Similarly, we have $t_{2i}=t^{\psi}_i$.

Now, the result follows by Lemma~\ref{phiplusstarequal}.
\end{proof}

Note that if $S \in \pv{ECom}$, Lemma~\ref{ephiAndpsiECom} holds for every $s,t\in S$. However, if $S \not\in \pv{ECom}$, it is possible that $\varphi^{(s,t)} \neq \psi^{(s,t)}$ (see Example~\ref{Ex012}). 

\begin{lem}\label{sftphi}
Let $s,t\in S$ and $e,f \in E(S)$ with \(t^{+} \leq e\) and \(s^{\ast} \leq f\).
We have \[\varphi^{(s,t)} = \varphi^{(s,ft)} \ \text{and}\ \psi^{(s,t)} = \psi^{(se,t)}.\]
\end{lem}

\begin{proof}
Let \(s^{\varphi}_0 = s, t^{\varphi}_0 =  t, {s'}^{\varphi}_0 = s\), and \({t'}^{\varphi}_0 =  ft\).
As \(s^{\ast} \leq f\), we have \(s^{\varphi}_1=s(s^{\ast}t)^{+}=s(s^{\ast}ft)^{+}={s'}^{\varphi}_1\) and \(t^{\varphi}_1=s^{\ast}t=s^{\ast}ft={t'}^{\varphi}_1\).
It follows that \(\varphi^{(s,t)} = \varphi^{(s,ft)}\).

Similarly, we have \(\psi^{(s,t)} = \psi^{(se,t)}\).
\end{proof}

In Lemma~\ref{sftphi}, it is noteworthy that 
\(\varphi^{(s,t)}\) might not equal \(\varphi^{(se,t)}\) 
or 
\(\psi^{(s,t)}\) might differ from \(\psi^{(s,ft)}\) (see Example~\ref{Ex013}).

By Lemma~\ref{sftphi}, the following corollary holds.

\begin{cor}\label{sPrimeAndsPrimePrime1} 
Let \(s' \ll s\) and \(t' \ll t\) with \({s'}^{\ast} = {t'}^{+}\).
We have \[\varphi^{(s',t{t'}^{\ast})} = \psi^{({s'}^{+}s,t')} = {s'}^{\ast}.\]
\end{cor}

\begin{lem}\label{phittP} 
Let \(s'\ll s\) and \(t'\ll t\).
If \(\varphi({s'}^{+}s,t{t'}^{\ast}) = {t'}^{+}\) then \(\varphi({s'}^{+}s,t') = {t'}^{+}\), 
and
if \(\psi({s'}^{+}s,t{t'}^{\ast}) = {s'}^{\ast}\) then \(\psi(s',t{t'}^{\ast}) = {s'}^{\ast}\).
\end{lem}

\begin{proof}
Suppose that \((({s'}^{+}s)^{\ast}t{t'}^{\ast})^{+} = {t'}^{+}\).
As \(({s'}^{+}s)^{\ast} \in \varphi^{+}(({s'}^{+}s)^{\ast}t{t'}^{\ast})\) and \((({s'}^{+}s)^{\ast}t{t'}^{\ast})^{+} = {t'}^{+}\), 
we have 
\[
({s'}^{+}s)^{\ast}t{t'}^{\ast}
=
{t'}^{+}({s'}^{+}s)^{\ast}t{t'}^{\ast}
=
({s'}^{+}s)^{\ast}{t'}^{+}t{t'}^{\ast}
=
({s'}^{+}s)^{\ast}t'.
\]
It follows that \((({s'}^{+}s)^{\ast}t')^{+} = {t'}^{+}\).

Similarly, the second statement holds.
\end{proof}

Note that in Lemma~\ref{phittP}, it may be that \(\varphi({s'}^{+}s,t') = {t'}^{+}\) and \(\varphi({s'}^{+}s,t{t'}^{\ast}) \neq {t'}^{+}\) (see Example~\ref{Ex014}). Additionally, the converse of the lemma may not hold true for the function \(\psi\).

%
%


\section{Study of the determinant of $\ll$--transitive singleton-rich semigroups}\label{DetNotInEcom}

As mentioned in Section~\ref{llAndlll}, we consider a finite semigroup 
$S$ with the assumption that the semigroup algebra $\mathbb{C}S$ is a unital algebra. 
In this section, we examine and study the determinant of $\ll$-transitive singleton-rich semigroups and compute the determinant of a class of these semigroups, which we subsequently define and refer to as $\ll$-smooth semigroups.

The following lemma could be useful for determining the $\ll$-transitive semigroups.

\begin{lem}
Suppose that $ef \in E(S)$, for every $e,f\in E(S)$. Then, the relation $\ll$ is transitive.
\end{lem}

\begin{proof}
Suppose the contrary that the relation $\ll$ is not transitive. 
Then, there exists a sequences 
\(s\ll s'\ll s''\)
with \(s\not\ll  s''\).
Hence, we get that
\(
s=s^{+} {s'}^{+} s'' {s'}^{\ast} s^{\ast}.
\)
By the assumption of the lemma, the elements $s^{+} {s'}^{+}$ and ${s'}^{\ast} s^{\ast}$ are idempotent.
Then, we have $s^{+} {s'}^{+} \leq s^{+}$ and ${s'}^{\ast} s^{\ast} \leq s^{\ast}$ and, thus, $s \ll s''$.
A contradiction with the assumption.
\end{proof}

Note, that if $\ll$ is transitive, it may be the case that $S$ contains idempotent $e$ and $f$ such that $ef$ is not idempotent (see Example~\ref{Ex017}).

We define the multiplication
$\sharp \colon \mathbb{C}S \times \mathbb{C}S \rightarrow \mathbb{C}S$ 
as follows:\\
\begin{equation*}
s\sharp t = 
\begin{cases}
  st,& \text{if}\ s^{+}=(st)^{+}, t^{\ast}=(st)^{\ast}\ \text{and}\ s^{\ast}=t^{+};\\ 
  0,& \text{otherwise},
\end{cases}
\end{equation*}
for every $s,t\in S$. 
Also, we define the multiplication
$\sharp\limits^{e} \colon \mathbb{C}S \times \mathbb{C}S \rightarrow \mathbb{C}S$, for some $e\in E(S)$,
as follows:\\
\begin{equation*}
s\sharp^{e} t = 
\begin{cases}
  st & \text{if}\ s^{+}=(st)^{+}, t^{\ast}=(st)^{\ast}\ \text{and}\ s^{\ast}=t^{+}=e;\\
  0& \text{otherwise}.
\end{cases}
\end{equation*}

According to Proposition~\ref{Z}, the mapping $Z$ is bijective. 
For the case the relation $\ll$ is transitive on $S$,
we define the following multiplication on $\mathbb{C}S \times \mathbb{C}S$
\begin{equation*}
Z(s)\boldsymbol{*}Z(t) = 
\sum\limits_{\substack{s'\ll s,\\ t'\ll t}}
s'\stackrel{\varphi^{({s'}^{+}s,t{t'}^{\ast})}}{\sharp} t'.
\end{equation*}
By applying M\"{o}bius inversion, we have 
\[
u= \displaystyle\sum\limits_{u'\ll u}\mu_S(u', u)Z(u')\
\text{and}\ 
v= \displaystyle\sum\limits_{v'\ll v}\mu_S(v', v)Z(v'),\]
 for every $u,v \in S$. 
Then \[u\boldsymbol{*}v=\displaystyle\sum\limits_{u'\ll u,v'\ll v}\mu_S(u', u)\mu_S(v', v)Z(u')\boldsymbol{*}Z(v').\]

\begin{prop}\label{s*tHomIsTransitive}
We have $Z(s)\boldsymbol{*}Z(t)= Z(st)$, for all $s,t\in S$.
\end{prop}

\begin{proof}
Let $s'\ll s$ and $t'\ll t$. 
The first step is to prove that $s't'\ll st$, if $s'\stackrel{\varphi^{({s'}^{+}s,t{t'}^{\ast})}}{\sharp} t'\neq 0$.
Then, the following conditions hold:
\begin{equation}\label{ConditionsOfs't'2}
{s'}^{+}=(s't')^{+}, {t'}^{\ast}=(s't')^{\ast}\ \text{and}\ {s'}^{\ast}={t'}^{+}.
\end{equation}
Since $s'\ll s$ and $t'\ll t$, we have $s' = {s'}^{+}s{s'}^{\ast}$ and $t' = {t'}^{+}t{t'}^{\ast}$.
Now, as $s'\stackrel{\varphi^{({s'}^{+}s,t{t'}^{\ast})}}{\sharp} t'\neq 0$, we have ${s'}^{\ast}={t'}^{+}=\varphi^{({s'}^{+}s,t{t'}^{\ast})}$.
Hence, by Lemma~\ref{ephiAndpsi}, we have 
\[s't' = {s'}^{+}s{s'}^{\ast}{t'}^{+}t{t'}^{\ast} = {s'}^{+}s\varphi^{({s'}^{+}s,t{t'}^{\ast})}t{t'}^{\ast} = {s'}^{+}st{t'}^{\ast}.\]
As ${s'}^{+}=(s't')^{+}$ and ${t'}^{\ast}=(s't')^{\ast}$, we have $s't'\ll st$.

Secondly, we need to show that if $u\ll st$, for some $u$ in $S$, there exists a unique pair $(s^{\circ},t^{\circ})$ such that $s^{\circ}\ll s$, $t^{\circ}\ll t$,  $s^{\circ}\stackrel{\varphi^{({s^{\circ}}^{+}s,t{t^{\circ}}^{\ast})}}{\sharp} t^{\circ}\neq 0$, and $s^{\circ}t^{\circ}=u$.
Let $\varphi = \varphi^{(u^{+} s ,t u^{\ast})}$, 
$s_1= u^{+} s \varphi$,
and
$t_1= \varphi t u^{\ast}$. 
Since
\(
u^{+}s \varphi tu^{\ast}
=
u
,
\)
it easily follows that $s_1^{+} = u^{+}$ and $t_1^{\ast} = u^{\ast}$.
By Lemma~\ref{phiplusstarequal}, we deduce $s_1^{\ast}=t_1^{+}=\varphi$.
Then, we get that
\(s_1 \ll s\ \text{and}\
  t_1 \ll t.\)
Also, we have 
\begin{align*}
s_1
\stackrel{\varphi^{
(
s_1^{+}s,
tt_1^{\ast}
)
}}{\sharp}
t_1
&
=
s_1
\stackrel{\varphi^{
(
u^{+}s,
tu^{\ast}
)
}}{\sharp}
t_1
=
s_1
\stackrel{\varphi}{\sharp}
t_1
=
s_1t_1
=
u.
\end{align*}

There is then our desired pair. 
Now, we prove the uniqueness of this existence. 
Let $s_1,s_2,t_1,t_2\in S$ such that $s_1,s_2\ll s$, $t_1,t_2\ll t$, $s_1t_1=s_2t_2=u$, and 
$s_1\stackrel{\varphi^{({s_1}^{+}s,t{t_1}^{\ast})}}{\sharp} t_1,s_2\stackrel{\varphi^{({s_2}^{+}s,t{t_2}^{\ast})}}{\sharp} t_2\neq 0$.
Then, the pairs $(s_1,t_1)$ and $(s_2,t_2)$ satisfy conditions~(\ref{ConditionsOfs't'2}) and, thus, we have
 $s_1^{+}=s_2^{+}=u^{+}$ and $t_1^{\ast}=t_2^{\ast}=u^{\ast}$. 
Also, we have $s_1^{\ast}=s_2^{\ast}=t_1^{+}=t_2^{+}=\varphi^{(u^{+}s,tu^{\ast})}$.
Therefore, we have $s_1=u^{+}s\varphi^{(u^{+}s,tu^{\ast})}=s_2$ and $t_1=\varphi^{(u^{+}s,tu^{\ast})}su^{\ast}=t_2$. 

The result follows.
\end{proof}

\begin{thm}\label{Zisom}
Suppose that the relation $\ll$ on $S$ is transitive.
The mapping $Z$ is an isomorphism of $\mathbb{C}$-algebras. 
\end{thm}

\begin{proof}
By Propositions~\ref{Z} and~\ref{s*tHomIsTransitive}, the result follows.
\end{proof}


Let $s,t\in S$.
We have \[
s= \displaystyle\sum\limits_{s'\ll s}\mu_S(s', s)Z(s')\
\text{and}\ 
t= \displaystyle\sum\limits_{t'\ll t}\mu_S(t', t)Z(t').\]
Then, we get that
\begin{align*}
s\boldsymbol{*}t
&=
\displaystyle\sum\limits_{\substack{s'\ll s,\\ t'\ll t}}\mu_S(s', s)\mu_S(t', t)Z(s')\boldsymbol{*}Z(t')\\
&=
\displaystyle\sum\limits_{\substack{s'\ll s,\\ t'\ll t}}\mu_S(s', s)\mu_S(t', t)
\Big(
\sum\limits_{\substack{s''\ll s',\\ t''\ll t'}}
s''\stackrel{\varphi^{({s''}^{+}s',t'{t''}^{\ast})}}{\sharp} t''
\Big)\\
&=
\sum\limits_{\substack{s''\ll s,\\ t''\ll t}}\big[\sum\limits_{\substack{s''\ll s'\ll s,\\ t''\ll t'\ll t,\\ 
s''\stackrel{\varphi^{({s''}^{+}s',t'{t''}^{\ast})}}{\sharp} t''\neq 0}}
\mu_S(s', s)\mu_S(t', t)
\big]
s''t''\\
&=
\sum\limits_{\substack{s''\ll s,\\ t''\ll t}}
\big[
\sum\limits_{s''\ll s'\ll s}
\big(
\sum\limits_{\substack{t''\ll t'\ll t,\\ s''\stackrel{\varphi^{({s''}^{+}s',t'{t''}^{\ast})}}{\sharp} t''\neq 0}}
\mu_S(t', t)
\big)
\mu_S(s', s)
\big]
s''t''
.
\end{align*}
We define the function \(\xi\colon S^{(s)}\times S^{(t)}\rightarrow \mathbb{C}\), 
where \(S^{(s)}=\{s'' \in S \mid s'' \ll s\}\), for every \(s \in S\), as follows:
\begin{align*}
\xi(s'',t'') 
&=
\sum\limits_{s''\ll s'\ll s}
\big(
\sum\limits_{\substack{t''\ll t'\ll t,\\ s''\stackrel{\varphi^{({s''}^{+}s',t'{t''}^{\ast})}}{\sharp} t'' \neq 0}}
\mu_S(t', t)
\big)
\mu_S(s', s),
\end{align*}
for every \(s'' \in S^{(s)}\) and \(t'' \in S^{(t)}\).

In the continuation of this section, we focus on a class of $\ll$--transitive singleton-rich semigroups called $\ll$-smooth and compute their determinants. 

\begin{deff}\label{ConditionsnEq7}   
Let \(S\) be a $\ll$--transitive singleton-rich semigroup. 
We say that $S$ is $\ll$-smooth if for every sequences                           
\(s''\ll s' \ll s\) and \(t''\ll t' \ll t\),
the following statements hold:
\begin{enumerate}
\item 
if \(s''\sharp t''\neq 0\), then we have \(\varphi^{({s''}^{+}s',t'{t''}^{\ast})} = \varphi({s''}^{+}s',t'{t''}^{\ast}).\) 
\item if \(s''\sharp t''\neq 0\), then we have \(({s''}^{+}s')^{\ast}t'{t''}^{\ast} = t''\) if and only if
 \[({s''}^{+}s')^{\ast}t{t''}^{\ast} = t''.\]
\item if \(s''({s''}^{+}s)^{\ast}=s''\), then we have \(s''({s''}^{+}s')^{\ast}=s''\).
\end{enumerate}
\end{deff}


The following lemma straightforwardly follows from the definition of $\ll$-smooth semigroups.

\begin{lem}\label{ConditionsnEq8} 
Let $S$ be a $\ll$-smooth semigroup and let \(s''\ll s' \ll s\) and \(t''\ll t' \ll t\) be sequences in $S$. If \(s''\sharp t''\neq 0\), then
the following statements hold:
\begin{enumerate}
\item we have \(\varphi({s''}^{+}s',t{t''}^{\ast}) = {t''}^{+}\) if and only if \(\varphi({s''}^{+}s',t'{t''}^{\ast}) = {t''}^{+}\).
\item 
if \(s''\ll s_1' \ll s'_2 \ll s\) then 
\(\varphi({s''}^{+}s'_2,t'') = {t''}^{+}\) implies \(\varphi({s''}^{+}s'_1,t'') = {t''}^{+}\).
\end{enumerate}
\end{lem}

%
%

An example showing that Lemma~\ref{ConditionsnEq8} does not necessarily hold can be found in the semigroup \( S \) from Example~\ref{Ex018}. In this example, the relation \( \ll \) is not transitive in the semigroup \( S \).

\begin{lem}\label{LemmaMain}
Suppose that \(S\) is $\ll$-smooth.
Let \(s,t\in S\) with \(s^{\ast} \neq t^{+}\) and \(s'' \in S^{(s)}\) and \(t'' \in S^{(t)}\) with \(s'' \sharp t'' \neq 0\).
We have \(\xi(s'',t'') \neq 0\), if and only if the following conditions hold:
\begin{enumerate}
\item \(s'' = st^{+}\), \(t''= t\);
\item \({s''}^{+}=s^{+}\);
\item \( \xi(st^{+},t) = \sum\limits_{\substack{st^{+}\ll s'\ll s,\\ t^{+} \leq (s^{+}s')^{\ast}}}\mu_S(s', s) \neq 0.\)
\end{enumerate}
Moreover, we have \(s^{\ast} \not\leq t^{+}\). Furthermore, if \(s \ast t \neq 0\), then we have \[s \ast t =\big(\sum\limits_{\substack{st^{+}\ll s'\ll s,\\ t^{+} \leq (s^{+}s')^{\ast}}}\mu_S(s', s)\big)st.\]
\end{lem}

\begin{proof}
First, suppose that \(\xi(s'',t'') \neq 0\). We first prove conditions (1) and (2), namely that \(s''=st^{+}\), \(t''=t\),
 and \({s''}^{+}=s^{+}\). We then compute
\(\xi(st^{+},t)\), which gives condition (3).

Assume that \(t'' \neq t\). 
Let \(s'\) be an element of \(S\) with \(s''\ll s'\ll s\). 
As \(S\) is $\ll$-smooth and \(s'' \sharp t'' \neq 0\), we have 
\(\varphi^{({s''}^{+}s',t'{t''}^{\ast})} = \varphi({s''}^{+}s',t'{t''}^{\ast})\), for every \(t'\) with \(t''\ll t'\ll t\).
Let \(\Delta = \sum\limits_{\substack{t''\ll t'\ll t,\\ s''\stackrel{\varphi({s''}^{+}s',t'{t''}^{\ast})}{\sharp} t'' \neq 0}}
\mu_S(t', t)\).
By Lemma~\ref{ConditionsnEq8}.(1), if \(\varphi({s''}^{+}s',t'') \neq {s''}^{\ast}\), then \(\varphi({s''}^{+}s',t'{t''}^{\ast}) \neq {s''}^{\ast}\), for every \(t''\ll t'\ll t\) and, thus \(\Delta = 0\).
Also, if \(\varphi({s''}^{+}s',t'') = {s''}^{\ast}\), then, we have \(\varphi({s''}^{+}s',t'{t''}^{\ast}) = {s''}^{\ast}\), for every \(t''\ll t'\ll t\), and, thus, we have 
\(\Delta=
\sum\limits_{\substack{t''\ll t'\ll t}}
\mu_S(t', t)=0.\)
It follows that \(\xi(s'',t'') = 0\), a contradiction.

Then, \(t'' = t\). We get that
\begin{align*}
\xi(s'',t) 
&=
\sum\limits_{\substack{s''\ll s'\ll s,\\ s''\stackrel{\varphi({s''}^{+}s',t)}{\sharp} t \neq 0}}
\mu_S(s', s).
\end{align*}

As \(s^{\ast} \neq t^{+}\) and \(s'' \sharp t \neq 0\), we have \(s'' \neq s\).

If \(s''\stackrel{\varphi({s''}^{+}s',t)}{\sharp} t \neq 0\), for all \(s''\ll s'\ll s\), and considering that \({s''} \neq s\), it follows that \(\sum\limits_{s''\ll s'\ll s}\mu_S(s', s) = 0\).

Then, there is an element \(s''\ll x\ll s\) such that \(\varphi({s''}^{+}x,t) \neq {s''}^{\ast}\).
Let \(s''\ll x_1,\ldots, x_{n} \ll s\) be the minimal elements with respect to the relation \(\ll\) satisfying \(\varphi({s''}^{+}x_i,t) \neq {s''}^{\ast}\).
By Lemma~\ref{ConditionsnEq8}.(2), 
if \(x_i \ll u\), for some \(s''\ll u\ll s\), then \(\varphi({s''}^{+}u,t) \neq {s''}^{\ast}\).

For, each \(s''\ll s'\ll s\), define  
\[
X_{s'} = \{s'' \ll x \ll s \mid s' \ll x \}.
\]  
Then, by the Inclusion-Exclusion Principle, we have 
\begin{align*}
\sum\limits_{\substack{s''\ll x\ll s\\ \varphi({s''}^{+}x,t) \neq {s''}^{\ast}}}\mu_S(x, s)
=&
\sum\limits_{x\in \bigcup_{1\leq i\leq n} X_{x_i}}\mu_S(x, s)\\
=&
\sum\limits_{i=1}^{n}\sum\limits_{x\in X_{x_{i}}}\mu_S(x, s)
-\sum\limits_{\substack{1\leq i_1<i_2\leq n}}\sum\limits_{\substack{s''\ll x\ll s\\ x\in X_{x_{i_1}}\bigcap X_{x_{i_2}}}}\mu_S(x, s)\\
&+\sum\limits_{\substack{1\leq i_1<i_2<i_3\leq n}}\sum\limits_{\substack{s''\ll x\ll s\\ x\in X_{x_{i_1}}\bigcap X_{x_{i_2}}\bigcap X_{x_{i_3}}}}\mu_S(x, s)
-\cdots\\
&\pm\sum\limits_{\substack{s''\ll x\ll s\\ x\in X_{x_{1}}\bigcap \cdots \bigcap X_{x_{n}}}}\mu_S(x, s).
\end{align*}

We prove \( s = {s''}^{+}s \) by contradiction. 
Assume, for the sake of contradiction, that \( s \neq {s''}^{+}s \). 

We proceed by induction on the size of the set \( \bigcup_{1\leq i\leq n} X_{x_i} \) that  
\[
\sum\limits_{x\in \bigcup_{1\leq i\leq n} X_{x_i}}\mu_S(x, s) = 0.
\] 

Suppose that \( {s''}^{+}x_i \neq x_i \), for some \( 1\leq i\leq n \).  
We have \({s''}^{+}x_i \ll x_i\).
Since \( x_i \) is minimal with respect to \( \ll \) such that \( \varphi({s''}^{+}x_i,t) \neq {s''}^{\ast} \), we obtain  
\[
\varphi({s''}^{+}{s''}^{+}x_i,t) = {s''}^{\ast},
\]  
which leads to a contradiction.  

Therefore, we conclude that \( {s''}^{+}x_i = x_i \) for all \( 1\leq i\leq n \).

Since \( x_i \ll s \), we have  
\(
x_i = x_i^{+} s x_i^{\ast}.
\)  
Moreover, as \( {s''}^{+}x_i = x_i \), it follows that \( x^{+}_i \leq {s''}^{+} \).  
Thus, we obtain  
\(
x_i = x_i^{+} {s''}^{+} s x_i^{\ast},
\)  
which implies that  
\(
x_i \ll {s''}^{+}s \ll s,
\)
for every \( 1\leq i\leq n \).

Hence, each subset \( X_{x_i} \) contains two distinct elements, \( s \) and \( {s''}^{+}s \).

Our inductive hypothesis assumes that for every subset \( \bigcup_{1\leq i\leq n} X_{x_i} \) of size at most $k$, where each $X_{x_i}$ contains two elements, namely \( s \) and \( {s''}^{+}s \), and satisfies \( {s''}^{+}x_i = x_i \) for all \( 1\leq i\leq n \), we have
\[
\sum\limits_{x\in \bigcup_{1\leq i\leq n} X_{x_i}}\mu_S(x, s) = 0.
\] 

We establish the base case of our induction when \( \abs{\bigcup_{1\leq i\leq n} X_{x_i}} = 2 \).
In this case, the set \( \bigcup_{1\leq i\leq n} X_{x_i} \) consists of exactly two elements, namely \( s \) and \( {s''}^{+}s \).
Thus, we immediately obtain:
\[
\sum\limits_{x\in \bigcup_{1\leq i\leq n} X_{x_i}}\mu_S(x, s) =\sum\limits_{x\in X_{{s''}^{+}s} }\mu_S(x, s)= 0.
\] 

For the inductive step, we assume that \( \abs{\bigcup_{1\leq i\leq n} X_{x_i}}=k+1 \).

If $n=1$, then we have 
\begin{align*}
\sum\limits_{\substack{s''\ll x\ll s\\ \varphi({s''}^{+}x,t) \neq {s''}^{\ast}}}\mu_S(x, s)
=&
\sum\limits_{x\in X_{x_1}}\mu_S(x, s)=0,
\end{align*}
since the subset $ X_{x_1}$ contains two distinct elements: \( s \) and \( {s''}^{+}s \).

Otherwise, \( n > 1 \).  

Let \( 1\leq i_1,\ldots, i_m\leq n \).  
The subset \( X_{x_{i_1}} \cap \cdots \cap X_{x_{i_m}} \) contains minimal elements \( y_1, \ldots, y_l \) with respect to the relation \( \ll \);  
\[
X_{x_{i_1}} \cap \cdots \cap X_{x_{i_m}} =\bigcup_{1\leq i\leq l}X_{y_i}.
\]  
Suppose that for some \( 1 \leq i \leq l \), we have \( {s''}^{+}y_i \neq y_i \).  
Since \( x_{i_k} \ll y_i \) and \( {s''}^{+}x_{i_k} = x_{i_k} \) for every \( i_k \in \{i_1, \ldots, i_m\} \), we obtain  
\[
x_{i_k} = x_{i_k}^{+} y_i x_{i_k}^{\ast} = x_{i_k}^{+} {s''}^{+} y_i x_{i_k}^{\ast}.
\]  
Thus, \( x_{i_k} \ll {s''}^{+}y_i \ll y_i\) for every \( i_k \in \{i_1, \ldots, i_m\} \), contradicting the minimality of \( y_i \) in \( X_{x_{i_1}} \cap \cdots \cap X_{x_{i_m}} \).  
Therefore, we must have \( {s''}^{+}y_i = y_i \) for all \( 1 \leq i \leq l \).  
Moreover, we obtain  
\[
y_i \ll {s''}^{+}s \ll s,
\]
for every \( 1\leq i\leq l \).  

Now, as \( n\neq 1 \), the size of  
\[
X_{x_{i_1}} \cap \cdots \cap X_{x_{i_m}} =\bigcup_{1\leq i\leq l}X_{y_i}
\]  
is strictly less than \( k+1 \), and by the inductive hypothesis, we have  
\[
\sum\limits_{x\in X_{x_{i_1}} \cap \cdots \cap X_{x_{i_m}}}\mu_S(x, s) = 0.
\]

Then, by the Inclusion-Exclusion Principle, we obtain  
\[
\sum\limits_{\substack{s''\ll x\ll s\\ \varphi({s''}^{+}x,t) \neq {s''}^{\ast}}}\mu_S(x, s) = 0.
\]

Thus,  
\[
\xi(s'', t'') = 0,
\]  
which contradicts our assumption. Hence, we must have:  
\[
s = {s''}^{+}s.
\]

It follows that \(s^{+}\leq {s''}^{+}\).
Also, as \(s'' \ll s\), we have
\[
s''=s''^{+}ss''^{\ast}=ss''^{\ast},
\]
and, thus, \({s''}^{+} \leq s^{+}\). Therefore,
\[
{s''}^{+}=s^{+},
\]
which proves condition~(2).

Moreover, since \(s''\sharp t\neq 0\), we have
\({s''}^{\ast}=t^{+}\). Hence,
\[
s''=ss''^{\ast}=st^{+}.
\]
Together with \(t''=t\), this proves condition~(1).



It remains to establish condition~(3).
Hence, we have  
\[
\xi(st^{+},t) 
=
\sum\limits_{\substack{st^{+}\ll s'\ll s,\\ st^{+}\stackrel{\varphi((st^{+})^{+}s',t)}{\sharp} t \neq 0}}
\mu_S(s', s)
=
\sum\limits_{\substack{st^{+}\ll s'\ll s,\\ st^{+}\stackrel{\varphi(s^{+}s',t)}{\sharp} t \neq 0}}
\mu_S(s', s).
\]
Rewriting the condition, we obtain:  
\[
\xi(st^{+},t) 
=
\sum\limits_{\substack{st^{+}\ll s'\ll s,\\ \varphi(s^{+}s',t)= t^{+}}}
\mu_S(s', s)
=
\sum\limits_{\substack{st^{+}\ll s'\ll s,\\ t^{+} \leq (s^{+}s')^{\ast}}}
\mu_S(s', s).
\]
In the second part, we use the equivalence of the conditions \( \varphi(s^{+}s',t) = t^{+} \) and \( t^{+} \leq (s^{+}s')^{\ast} \).
Since \(\xi(s'',t'')=\xi(st^{+},t)\neq 0\), the above sum is nonzero.
Thus condition~(3) follows.


Now, suppose that conditions
(1), (2) and (3) of the lemma hold. In particular,
\(s''=st^{+}\) and \(t''=t\).

We have 
\begin{align*}
\xi(s'',t'') 
=
\xi(st^{+},t)
&=
\sum\limits_{\substack{st^{+}\ll s'\ll s,\\ s''\stackrel{\varphi({s}^{+}s',t)}{\sharp} t \neq 0}}
\mu_S(s', s)
=
\sum\limits_{\substack{st^{+}\ll s'\ll s,\\ t^{+} \leq (s^{+}s')^{\ast}}}
\mu_S(s', s).
\end{align*}

Hence, we have \(\xi(s'',t'') \neq 0\).

As \(st^{+} \sharp t \neq 0\) and \(s^{\ast} \neq t^{+}\), we have \(s^{\ast} \not\leq t^{+}\).

Moreover, we deduce that \(s \ast t = \xi(st^{+},t)st^{+}t = \big(\sum\limits_{\substack{st^{+}\ll s'\ll s,\\ t^{+} \leq (s^{+}s')^{\ast}}}
\mu_S(s', s)\big)st\).
\end{proof}

Suppose that \( S \) is a \(\ll\)-smooth semigroup. Furthermore, assume that for every sequence \( s'' \ll s' \ll s \), if  
\[
s''({s''}^{+}s_1)^{\ast} \neq s'' \quad \text{and} \quad s''({s''}^{+}s_2)^{\ast} \neq s'',
\]  
for some elements \( s'' \ll s_1, s_2 \ll s \), then there exists an element \( s'' \ll s_3 \ll s \) such that  
\[
s_3 \ll s_1, s_2 \quad \text{and} \quad s''({s''}^{+}s_3)^{\ast} \neq s''.
\]  

This condition is equivalent to the following: if \( s'' \ll s_1', s_2' \ll s \) and  
\[
\varphi({s''}^{+}s_1',t), \varphi({s''}^{+}s_2',t) \neq {t}^{+},
\]  
then there exists an element \( s'' \ll s_3' \ll s \) such that  
\begin{equation}\label{imposedCondition}
s_3' \ll s_1', s_2' \quad \text{and} \quad \varphi({s''}^{+}s_3',t) \neq {t}^{+},
\end{equation}
for every \( t \in S \).

In Lemma~\ref{LemmaMain}, if \( S \) satisfies Condition~(\ref{imposedCondition}), \( \xi(s'',t'') \neq 0 \), and the other conditions of the lemma hold, we can establish that  
\[
\xi(s'',t'') = -1.
\]
Indeed, as we discussed in the proof of Lemma~\ref{LemmaMain}, there exists an element \( st^{+} \ll x \ll s \) such that \( \varphi({s''}^{+}x,t) \neq {s''}^{\ast} \).  
By Condition~(\ref{imposedCondition}), there is a minimal element \( u \) with respect to the relation \( \ll \) such that \( \varphi({s''}^{+}u,t) \neq {s''}^{\ast} \).  

If \(u \neq s\), then we have 
\[
\sum\limits_{\substack{s''\ll s'\ll s,\\ s''\stackrel{\varphi({s''}^{+}s',t)}{\sharp} t = 0}}\mu_S(s', s)
=
\sum\limits_{u\ll s'\ll s}\mu_S(s', s)=0
\] 
and, thus, we deduce
\(\xi(s'',t) = 0\).
Then, we have \(u=s\) and thus, 
\[\xi(s'',t) 
= \sum\limits_{\substack{s''\ll s'\ll s,\\ s''\stackrel{\varphi({s''}^{+}s',t)}{\sharp} t \neq 0}}\mu_S(s', s) 
= \sum\limits_{s''\ll s'\ll s}\mu_S(s', s)-\mu_S(s, s)
= -1.
\]


\begin{cor}\label{LemmaMain2}
Let \( S \) be a \(\ll\)-smooth semigroup.  
Suppose \( s, t \in S \) such that \( s^{\ast} \neq t^{+} \) and \( s \ast t \neq 0 \).  
Then, for every \( t' \in S \) with \( {t'}^{+} = t^{+} \), we have the following equivalence:
\[
s \ast t' \neq 0 \quad \text{if and only if} \quad st^{+} \sharp t' \neq 0.
\]
Furthermore, if \( s \ast t' \neq 0 \), then it holds that
\(
s \ast t' = \big(\sum\limits_{\substack{st^{+}\ll s'\ll s,\\ t^{+} \leq (s^{+}s')^{\ast}}}\mu_S(s', s)\big) st'.
\)
\end{cor}

\begin{proof}
Let $t'\in S$ with \({t'}^{+} = t^{+}\).

First, suppose that \( s \ast t' \neq 0 \).  
Then, there exist elements \( s'' \ll s \) and \( t'' \ll t' \) such that \( s'' \sharp t'' \neq 0 \) and \( \xi(s'', t'') \neq 0 \).  
By Lemma~\ref{LemmaMain}, we know that \( s'' = s t'^{+} = s t^{+} \) and \( t'' = t' \).  
This implies that \( st^{+} \sharp t' \neq 0 \).

Now, suppose that \(s{t}^{+} \sharp t' \neq 0\). 

As \( s \ast t \neq 0 \) and \( s^{\ast} \neq t^{+} \), it follows from Lemma~\ref{LemmaMain} that \( st^{+} \sharp t \neq 0 \) and  
\[
\xi(st^{+},t) = \sum\limits_{\substack{st^{+}\ll s'\ll s,\\ t^{+} \leq (s^{+}s')^{\ast}}}\mu_S(s', s) \neq 0.
\]
for the function \( \xi\colon S^{(s)}\times S^{(t)}\rightarrow \mathbb{C} \).  

Similarly, for the function \( \xi'\colon S^{(s)}\times S^{(t')}\rightarrow \mathbb{C} \), which is defined analogously to \( \xi \), we obtain  
\[
\xi'(st^{+},t') = \sum\limits_{\substack{st^{+}\ll s'\ll s,\\ t^{+} \leq (s^{+}s')^{\ast}}}\mu_S(s', s) \neq 0.
\]
Therefore, by Lemma~\ref{LemmaMain}, we conclude that 
\[s \ast t' 
=
\xi'(st^{+},t')st'
=
\big(\sum\limits_{\substack{st^{+}\ll s'\ll s,\\ t^{+} \leq (s^{+}s')^{\ast}}}\mu_S(s', s)\big)st'
\neq 0.\]
\end{proof}

The following lemma holds with a similar proof to Lemma~\ref{LemmaMain}.

\begin{lem}\label{LemmaMain3}
Suppose that \(S\) is $\ll$-smooth.
Let \(s,t\in S\), 
\(s'' \in S^{(st^{+})}\) and \(t'' \in S^{(t)}\). 
We have \(\xi(s'',t'') \neq 0\), if and only if \(s'' = st^{+}\), \(t''= t\) and \(st^{+} \sharp t \neq 0\).
Furthermore, if \(st^{+} \sharp t \neq 0\), then we have \(st^{+} \ast t =st\). 
\end{lem}

Note that we may apply Lemma~\ref{LemmaMain3} in the case when \( s^{\ast} = t^{+} \), replacing \( st^{+} \) with \( s \) in the lemma.  
However, the \(\ll\)-smoothness condition of \( S \) is necessary.  
For instance, the semigroup given in Example 5.10 of \cite{Sha-Det} does not satisfy this lemma.

By Corollary~\ref{LemmaMain2} and Lemma~\ref{LemmaMain3}, the following corollary holds.

\begin{cor}\label{LemmaMain4}
Suppose that \(S\) is $\ll$-smooth.
Let \(s,t\in S\) with \(s^{\ast} \neq t^{+}\) and \(s\ast t \neq 0\).
Then, \(st^{+}\ast t'\neq 0\) if and only if \(s\ast t'\neq 0\), for every \(t' \in S\) with \({t'}^{+} = t^{+}\).
Furthermore, if \(st^{+}\ast t'\neq 0\), then we have 
\[s\ast t'=\big(\sum\limits_{\substack{st^{+}\ll s'\ll s,\\ t^{+} \leq (s^{+}s')^{\ast}}}\mu_S(s', s)\big)st^{+}\ast t',\] 
for every \(t' \in S\) with \({t'}^{+} = t^{+}\).
\end{cor}

%
%

Let $X_e=\{x_s\in X\mid s\in \widetilde{L}_e\widetilde{R}_e\}$.
Let $\theta_e(X_e)$ be the determinant of the submatrix $\widetilde{L}_e\times\widetilde{R}_e$ of the Cayley table $(S,\boldsymbol{*})$, for every idempotent $e\in S$. 
Let \(M\) be a matrix obtained from \(C(X)\) over \((S,\boldsymbol{*})\) by rearranging and shifting its rows and columns so that, for every \(e\in E(S)\), the elements of the subsets \(\widetilde{L}_e\) and \(\widetilde{R}_e\) lie in adjacent rows and adjacent columns, respectively.
Let \(r_1,\ldots,r_{\abs{S}}\) and \(c_1,\ldots,c_{\abs{S}}\) denote the elements of \(S\) corresponding to the rows and columns of \(M\), respectively. Define
\[
R=(r_1,\ldots,r_{\abs{S}})
\quad\text{and}\quad
C=(c_1,\ldots,c_{\abs{S}})
\]
as the tuples corresponding to the rows and columns of \(M\), respectively.

Additionally, we define a matrix $M'$ as follows:
\begin{equation*}
[M'_{r_i,c_j}]  = \begin{cases}
  [M_{r_i,c_j}]& \text{if}\ r_i\in \widetilde{L}_e\ \text{and}\ c_j\in\widetilde{R}_{e}\ \text{for some idempotent}\ e;\\
  0& \text{otherwise},
\end{cases}
\end{equation*} for every \(1\leq i,j\leq \abs{S}\).

Let \(\mathbb{A}\) be the set of matrices of size \(\abs{S}\times\abs{S}\) whose rows and columns correspond to the tuples \(R\) and \(C\), respectively, with elements from the polynomial ring $\mathbb{C}[\widetilde{X}_S]$.
Let $A \in \mathbb{A}$.
We denote by \(A(r)\) the \(r\)-th row of the matrix \(A\).
We define a function \(\tau_A\colon R\rightarrow \mathcal{P}(R)\), for the matrix \(A\), as follows:
\[\tau_A(r)=\{rc^{+}\mid c\in C, r^{\ast} \neq c^{+}\ \text{and}\ [A_{r,c}]\neq 0\},\]
for every \(r\in R\), and let 
\[{R'}_A=\{r\in R\mid \tau_A(r)=\emptyset\}\ \text{and} \ {R''}_A=\{r\in R\mid \tau_A(r)\subseteq {R'}_A\}.\] 
Also, we define \[\tau_A^{(n)}(r)=\bigcup_{r'\in\tau_A^{(n-1)}(r)}\tau_A(r'),\]
for every integer $n>1$ with \(\tau_A^{(1)} = \tau_A\).

Additionally, we define a function \(\eta\colon \mathbb{A} \rightarrow \mathbb{A}\) as follows:\\
for each row \( r \) of \(A\):
\[\eta(A)(r) = A(r) - \sum\limits_{rc^{+}\in R'_A\bigcap \tau_A(r)} \big(\sum\limits_{\substack{rc^{+}\ll r'\ll r,\\ c^{+} \leq (r^{+}r')^{\ast}}}\mu_S(r', r)\big)A(rc^{+}),\]
for every \(A \in \mathbb{A}\). 

Let \(A \in \mathbb{A}\). 
We say that \(A\) is \(\tau\)-terminate, if for every \(r\in R\) and an integer \(n>0\), \(r\not\in\tau_A^{(n)}(r)\).
In this case, there exists an integer \(n_r\) such that \(\tau_A^{(n_r)}(r)=\emptyset\).
Otherwise, there exists a sequences \(c_1, c_2,\ldots\) of elements of \(S\) such that \(rc_1^{+}\cdots c_{i-1}^{+}c_i^{+} \in\tau_A(rc_1^{+}\cdots c_{i-1}^{+})\), for every $i>1$.
As $S$ is finite, there exist integers \(i<j\) such that \(rc_1^{+}\cdots c_i^{+}=rc_1^{+}\cdots c_i^{+}\cdots c_j^{+}\) and, thus, we have 
\[rc_1^{+}\cdots c_i^{+}=rc_1^{+}\cdots c_j^{+} \in \tau_A^{(j-i)}(rc_1^{+}\cdots c_i^{+}),\] a contradiction.

When \(A\) is \(\tau\)-terminate, the subset \({R'}_A\) is nonempty. Additionally, only one of the following conditions hold:
\begin{enumerate}
\item \(R'_A = R\);
\item \(R''_A \neq \emptyset\).
\end{enumerate}

\begin{lem}
Suppose that \(S\) is $\ll$-smooth.
We have \(\DDet M= \DDet M'\).
\end{lem}

\begin{proof}
First, we prove that the matrix $M$ is \(\tau\)-terminate.
Let \(r \in R\).
By Lemma~\ref{LemmaMain}, if \(rc^{+}\in \tau_M(r)\), for some $c \in C$, then we have \((rc^{+})^{\ast}=c^{+}\). 
Now, as \(r^{\ast} \neq c^{+}\), we have \(r\not\in\tau_M(r)\). 
Additionally, it follows that if \(rc_1^{+}\ldots c_{i-1}^{+}c_i^{+}\in \tau_M(rc_1^{+}\ldots c_{i-1}^{+})\), 
then we have \(rc_1^{+}\ldots c_{i-1}^{+} c_i^{+}\ll rc_1^{+}\ldots c_{i-1}^{+}\) with \(rc_1^{+}\ldots c_{i-1}^{+}c_i^{+}\neq rc_1^{+}\ldots c_{i-1}^{+}\).
Hence, by Proposition~\ref{lllES}.(2), the matrix \(M\) is \(\tau\)-terminate.

By Corollary~\ref{LemmaMain4}, 
\(\eta(M)\) is \(\tau\)-terminate, as \(\tau_{\eta(M)}(r) \subseteq \tau_M(r)\), for every \(r \in R\).
Similarly, \(\eta^{(i)}(M)\) is \(\tau\)-terminate, for every \(i>1\).
Also, we have \(R''_{M} \cup R'_{M}\subseteq R'_{\eta(M)}\) and \(R''_{\eta^{(i)}(M)} \cup R'_{\eta^{(i)}(M)}\subseteq R'_{\eta^{(i+1)}(M)}\), for every \(i>1\).
As \(\eta^{(i)}(M)\) is \(\tau\)-terminate, we have \(R'_{\eta^{(i)}(M)}=R\) or \(R''_{\eta^{(i)}(M)} \neq \emptyset\).
Hence, by repeatedly applying \(\eta\) to the matrix \(M\), there exists an integer \(n_M\) such that \(R'_{\eta^{(n_M)}(M)}=R\) and, thus,
\(\eta^{(n_M)}(M)=M'\). 

The result follows.
\end{proof}

Now as the determinant of the matrices \(M\) and \(M'\) are equal, 
by adapting the proof presented in \cite[Theorem 4.12]{Sha-Det} for the matrix \(M'\), we establish the following theorem for the factorization of the determinant of $S$.

\begin{thm}\label{Main-TheoremST2t}
Suppose that \(S\) is $\ll$-smooth.
For $s \in S$, put \[y_s=\sum\limits_{t\ll s}\mu_S(t, s)x_t.\] 
Then, we have
\[
\theta_S(X)
= \pm\prod\limits_{e\in E(S)}\widetilde{\theta}_{e}(Y_e)
\]
where $Y_e=\{y_s\mid s \in \widetilde{L}_e\widetilde{R}_e\}$.
Moreover, the determinant of $S$ is non-zero if and only if $\widetilde{\theta}_{e}(Y_e)\neq 0$, for every idempotent $e$.
\end{thm}

In Theorem~\ref{Main-TheoremST2t}, the sign of \(\prod\limits_{e\in E(S)}\widetilde{\theta}_{e}(Y_e)\) 
is contingent on whether the number of rearrangements and shifts applied to the rows and columns of $C(X)$ in order to construct the matrix \(M\)  is odd or even.
Example~\ref{Exa19} utilizes Theorem~\ref{Main-TheoremST2t} to demonstrate that the determinant of the semigroup $S_9$ is non-zero.

To calculate the entries of the submatrix \( \widetilde{L}_e \times \widetilde{R}_e \), extracted from the Cayley table \( (S, \boldsymbol{*}) \), we apply Lemma~\ref{LemmaMain3}.  
In particular, to determine the entry at position \( (s, t) \), it suffices to verify that  
\[
(st)^{+} = s^{+} \quad \text{and} \quad (st)^{\ast} = t^{\ast}.
\]  
If this condition holds, the corresponding entry is given by \( st \); otherwise, it is zero.   

Note that one could define the multiplication $\ast$ using the function $\psi$ instead of $\varphi$ and obtain analogous results for the function $\psi$ as well. 


%
%
%


\section*{Acknowledgments}
The author was partially supported by CMUP, member of LASI, which is financed by national funds through FCT -- Funda\c c\~ao para a Ci\^encia e a Tecnologia, I.P., under the projects with reference UIDB/00144/2020 and UIDP/00144/2020.
The author also acknowledges FCT support through a contract based on the “Lei do
Emprego Científico” (DL 57/2016).


\section{Appendix A}

In Appendix, we provide examples of the semigroups are not in $\pv{ECom}$ that the paper discusses.
We present detailed information for each semigroup, including its Cayley table and contracted semigroup determinants.
In the examples provided, we exclude the element zero from the rows and columns of the Cayley table.
If the multiplication of two non-zero elements in the Cayley table results in zero, we represent it with a dot.

\begin{example}\label{Ex01}
\ 
\begin{center}
\begin{tabular}{ c | c c c c c c } 
\textbf{\cc{$S_1$}} & $\cc{y}$ & $\cc{z}$ & $\cc{u}$ & $\cc{t}$ & $\cc{w}$ & $\cc{v}$\\ \hline 
$\cc{y}$  & . & . & . & . & . & $y$ \\ 
$\cc{z}$  & . & . & . & . & $z$  & .\\ 
$\cc{u}$  & . & . & . & $y$  & $u$  & $y$ \\ 
$\cc{t}$  & . & . & $z$  & . & $z$  & $t$ \\ 
$\cc{w}$  & . & $z$  & $z$  & $t$  & $w$  & $t$ \\ 
$\cc{v}$  & $y$  & . & $u$  & $y$  & $u$  & $v$ \\ 
\end{tabular}
\end{center}
$\DDet S_1 =-y^3z^3$.\\
The existence and nonexistence of the following relations illustrate that the relation $\ll$ in $S_1$ is not transitive:\\
 $z \ll u,t \ll w$ and $z \not\ll w$,\\
 $y \ll u,t \ll v$ and $y \not\ll v$.
\end{example}


\begin{example}\label{Ex00}
\ 
\begin{center}
\begin{tabular}{ c | c c c c } 
\textbf{\cc{$S_2$}} & $\cc{y}$ & $\cc{z}$ & $\cc{u}$ & $\cc{t}$\\ \hline 
$\cc{y}$  & . & . & . & $y$ \\ 
$\cc{z}$  & . & . & $z$  & .\\ 
$\cc{u}$  & $y$  & . & $u$  & .\\ 
$\cc{t}$  & . & $z$  & $z$  & $t$ \\ 
\end{tabular}
\end{center}
$\DDet S_2 =-y^2z^2$.\\
We have $\varphi^{(y, t)} (= t) \not\leq y^{+} (=u)$ and $\psi^{(t, y)} (= u) \not\leq y^{\ast} (=t)$.
\end{example}


\begin{example}\label{Ex012}
\ 
Consider the semigroup $S_2$ in Example~\ref{Ex00}.\\
We have $\varphi^{(y,u)} = t, \psi^{(y,u)} = x$ and $\varepsilon^{(y,u)}=(t,u)$.
\end{example}


\begin{example}\label{Ex013}
\ 
Consider the semigroup $S_2$ in Example~\ref{Ex00}.\\
We have \(\varphi^{(y,u)} \neq \varphi^{(yu,u)}\) and \(\psi^{(y,u)} \neq \psi^{(y,tu)}\).  
\end{example}


\begin{example}\label{Ex014}
\ 
\begin{center}
\begin{tabular}{ c | c c c c c } 
\textbf{\cc{$S_3$}} & $\cc{y}$ & $\cc{z}$ & $\cc{u}$ & $\cc{t}$ & $\cc{w}$\\ \hline 
$\cc{y}$  & . & . & . & $y$  & $y$ \\ 
$\cc{z}$  & . & . & $z$  & . & $z$ \\ 
$\cc{u}$  & $y$  & . & $u$  & . & $u$ \\ 
$\cc{t}$  & . & $z$  & $z$  & $t$  & $t$ \\ 
$\cc{w}$  & $y$  & $z$  & $u$  & $t$  & $w$ \\ 
\end{tabular}
\end{center}
$\DDet S_3 =y^2z^2(t + u - w - z)$.\\
We have \(w\ll w, z\ll w, \varphi(w^{+}w,uz^{\ast})\neq z^{+} (=t)\) and \(\varphi(w^{+}w,z)= z^{+}\).
\end{example}

\begin{example}\label{Ex017}
\ 
Consider the semigroup $S_2$ in Example~\ref{Ex00}.\\
The relation $\ll$ is transitive and the element $\vv{tu}$ is not idempotent.
\end{example}


\begin{example}\label{Ex018}
\ 
\begin{center}
\begin{tabular}{ c | c c c c c } 
\textbf{\cc{$S_4$}} & $\cc{y}$ & $\cc{z}$ & $\cc{u}$ & $\cc{t}$ & $\cc{w}$\\ \hline 
$\cc{y}$  & . & . & . & . & $y$ \\ 
$\cc{z}$  & . & . & . & . & $z$ \\ 
$\cc{u}$  & . & $y$  & . & $u$  & $y$ \\ 
$\cc{t}$  & . & $z$  & . & $t$  & $z$ \\ 
$\cc{w}$  & $y$  & $y$  & $u$  & $u$  & $w$ \\ 
\end{tabular}
\end{center}
$\DDet S_4 =0$.\\
We have \(y\ll u\ll w\) with \(y\not\ll w\), \(\varphi^{(y^{+}w,w)}= w^{+}\) and \(\varphi^{(y^{+}u,w)}\neq w^{+}\).
\end{example}


\begin{example}\label{Exa19}
\ 
\begin{center}
\begin{tabular}{ c | c c c c c c } 
\textbf{\cc{$S_5$}} & $\cc{y}$ & $\cc{z}$ & $\cc{u}$ & $\cc{t}$ & $\cc{w}$ & $\cc{v}$\\ \hline 
$\cc{y}$  & . & . & . & . & $y$  & $y$ \\ 
$\cc{z}$  & . & . & $z$  & $z$  & . & $z$ \\ 
$\cc{u}$  & $y$  & . & $u$  & $t$  & $y$  & $u$ \\ 
$\cc{t}$  & $y$  & . & $t$  & $u$  & $y$  & $t$ \\ 
$\cc{w}$  & . & $z$  & $z$  & $z$  & $w$  & $w$ \\ 
$\cc{v}$  & $y$  & $z$  & $u$  & $t$  & $w$  & $v$ 
\end{tabular}
\end{center}
$\DDet S_5 =-2y^2z^2(t - u)(u - v + w - y - z)$.\\
The table of the semigroup $Z(S_5)$ 
with the multiplication $\boldsymbol{*}$ is as follows:
\begin{center}
\begin{tabular}{ c | c c c c c c } 
\textbf{\cc{$(Z(S_5),\ast)$}} & $y$, & $z$, & $y+z+u$, & $y+z+t$, & $y+z+w$, & $y+z+u+w+v$\\ \hline 
$y$ & $.$ & $.$ & $.$ & $.$ & $y$ & $y$\\ 
$z$ & $.$ & $.$ & $z$ & $z$ & $.$ & $z$\\ 
$y+z+u$ & $y$ & $.$ & $y+z+u$ & $y+z+t$ & $y$ & $y+z+u$\\ 
$y+z+t$ & $y$ & $.$ & $y+z+t$ & $y+z+u$ & $y$ & $y+z+t$\\ 
$y+z+w$ & $.$ & $z$ & $z$ & $z$ & $y+z+w$ & $y+z+w$\\ 
$y+z+u+w+v$ & $y$ & $z$ & $y+z+u$ & $y+z+t$ & $y+z+w$ & $y+z+u+w+v$
\end{tabular}
\end{center}
By Theorem~\ref{Zisom}, the table on the left side can be used for $Z(S_5)$. 
\begin{center}
\begin{tabular}{ll} 
\begin{tabular}{ c | c c c c c c } 
\textbf{\cc{$(Z(S_5),\ast)$}} & $\cc{y}$ & $\cc{z}$ & $\cc{u}$ & $\cc{t}$ & $\cc{w}$ & $\cc{v}$\\ \hline 
$\cc{y}$ & . & . & . & . & $y$  & .\\ 
$\cc{z}$ & .  & . & $z$  & $z$ & . & .\\ 
$\cc{u}$ & $y$  & .  & $u$  & $t$ & $\qq{-y}$  & .\\ 
$\cc{t}$ & $y$  & .  & $t$  & $u$ & $\qq{-y}$  & .\\ 
$\cc{w}$ & .  & $z$ & $\qq{-z}$  & $\qq{-z}$  & $w$  & .\\ 
$\cc{v}$ & . & . & . & . & . & $v$ 
\end{tabular} &
\begin{tabular}{ c | c c c c c c } 
$M$ & $\cc{y}$ & $\cc{u}$ & $\cc{t}$ & $\cc{z}$ & $\cc{w}$ & $\cc{v}$\\ \hline 
$\cc{y}$ & . & . & . & . & $y$  & .\\ 
$\cc{w}$ & . & $\qq{-z}$  & $\qq{-z}$  & $z$  & $w$  & .\\ 
$\cc{z}$ & . & $z$  & $z$  & . & . & .\\ 
$\cc{u}$ & $y$  & $u$  & $t$  & . & $\qq{-y}$  & .\\ 
$\cc{t}$ & $y$  & $t$  & $u$  & . & $\qq{-y}$  & .\\ 
$\cc{v}$ & . & . & . & . & . & $v$ 
\end{tabular}
\end{tabular}
\end{center}
The table $M$ is on the right side that by rearranging and shifting the rows and columns of $Z(S_5)$ so that the elements of the subset $\widetilde{L}_e$ being adjacent rows and the elements of the subset $\widetilde{R}_{e}$ being adjacent columns for every idempotent $e\in E(S_5)$. 
We have ${R'}_M=\{y,z,v\}$ and ${R''}_M=\{w,u,t\}$. The table $M'=\eta(M)$ is as follows:
\begin{center}
\begin{tabular}{ c | c c c c c c } 
$M'$ & $\cc{y}$ & $\cc{u}$ & $\cc{t}$ & $\cc{z}$ & $\cc{w}$ & $\cc{v}$\\ \hline 
$\cc{y}$ & . & . & . & . & $y$  & .\\ 
$\cc{w}$ & . & .  & .  & $z$  & $w$  & .\\ 
$\cc{z}$ & . & $z$  & $z$  & . & . & .\\ 
$\cc{u}$ & $y$  & $u$  & $t$  & . & .  & .\\ 
$\cc{t}$ & $y$  & $t$  & $u$  & . & .  & .\\ 
$\cc{v}$ & . & . & . & . & . & $v$ 
\end{tabular}
\end{center}
Then, it is easy to compute that the determinant of the matrix $M'$ is non-zero and equal to $-2y^2z^2(t - u)v$. 
Consequently, the determinant of $\widetilde{\theta}_{S_5}(X_{S_5})$ is also non-zero, with a value of $-2y^2z^2(t - u)(u - v + w - y - z)$ where $v$ is substituted by the value $\sum\limits_{v'\ll v}\mu_S(v', v)v'=(u - v + w - y - z)$. 
\end{example}

\label{sec:appendix}

\bibliographystyle{plain}
\bibliography{ref-Det}

\end{document}